\documentclass[10pt]{article}

\usepackage[english]{babel}
\usepackage[utf8]{inputenc}
\usepackage[T1]{fontenc}

\usepackage{amsmath, amsthm, amssymb, mathtools}

\usepackage[shortlabels]{enumitem}

\usepackage{cases}

\usepackage[a4paper]{geometry}

\AtBeginDocument{\def\MR#1{}}

\usepackage{hyperref}

\providecommand{\keywords}[1]{\textbf{Keywords.} #1}
\providecommand{\MSC}[1]{\textbf{2010 Mathematics Subject Classification.} #1}

\newtheorem{theorem}{Theorem}[section]
\newtheorem{corollary}[theorem]{Corollary}

\newtheorem{proposition}[theorem]{Proposition}
\newtheorem{example}[theorem]{Example}

\theoremstyle{definition}
\newtheorem{definition}[theorem]{Definition}
\newtheorem{remark}[theorem]{Remark}

\renewcommand\epsilon{\varepsilon}
\renewcommand\mapsto{\longmapsto}

\usepackage{color}

\usepackage{xspace}

\newcommand{\R}{\field{R}\xspace}
\newcommand{\C}{\field{C}\xspace}
\newcommand{\N}{\field{N}\xspace}

\newcommand{\field}[1]{\ensuremath{\mathbb{#1}}}

\newcommand{\ens}[1]{ \left\{#1\right\} }

\newcommand\diag{\mathrm{diag} \xspace}

\newcommand{\Tinf}{T_{\mathrm{inf}}}

\newcommand\infT{T_{\mathrm{inf}}}
\newcommand\supT{T_{\mathrm{sup}}}

\newcommand\TRus{T_{\mathrm{[Rus]}}(\Lambda)}
\newcommand\TCN{T_{\mathrm{[CN]}}(\Lambda)}

\newcommand\pt[1]{\frac{\partial #1}{\partial t}}
\newcommand\px[1]{\frac{\partial #1}{\partial x}}
\newcommand\pxi[1]{\frac{\partial #1}{\partial \xi}}
\newcommand\pas[1]{\frac{\partial #1}{\partial s}}

\newcommand\Tau{\mathcal{T}}

\newcommand\ssin{s^{\mathrm{in}}}
\newcommand\ssout{s^{\mathrm{out}}}

\newcommand\ssinb{\bar{s}^{\mathrm{in}}}

\newcommand{\rank}{\mathrm{rank} \,}

\newcommand\Id{\mathrm{Id}}

\newcommand{\lin}[1]{\mathcal{L}(#1)}
\newcommand{\dom}[1]{ D(#1) }

\def\norm#1{\left\|#1\right\|}

\newcommand\st{\quad \middle| \quad}

\newcommand\equi{\, \sim \,}

\newcommand{\ps}[3]{ {\left\langle #1 , #2 \right\rangle}_{#3} }
\newcommand\ddt{\frac{d}{dt}}

\newcommand\setCN{\mathcal{B}}

\newcommand\critset{\Sigma}

\def\ds{\displaystyle}
\newcommand\dds{\frac{d}{ds}}

\usepackage{ifthen}
\newcommand{\syst}[2]{
\ifthenelse{\equal{#2}{}}{\left(\Lambda,#1,Q\right)}
{\ifthenelse{\equal{#2}{b}}{\left(\Lambda,-,Q,#1\right)}{}}
{\ifthenelse{\equal{#2}{c}}{\left(\Lambda,-,Q^0,#1\right)}{}}
}

\title{Equivalent one-dimensional first-order linear hyperbolic systems and range of the minimal null control time with respect to the internal coupling matrix}

\author{
Long Hu\thanks{School of Mathematics, Shandong University, Jinan, Shandong 250100, China.  E-mail: \texttt{hul@sdu.edu.cn}}
\and
Guillaume Olive\thanks{Faculty of Mathematics and Computer Science, Jagiellonian University, ul. {\L}ojasiewicza 6, 30-348 Krak\'{o}w, Poland. E-mail: \texttt{math.golive@gmail.com} or \texttt{guillaume.olive@uj.edu.pl}}
}

\date{September 16, 2021}

\begin{document}

\maketitle

\begin{abstract}
In this paper, we are interested in the minimal null control time of one-dimensional first-order linear hyperbolic systems by one-sided boundary controls.
Our main result is an explicit characterization of the smallest and largest values that this minimal null control time can take with respect to the internal coupling matrix.
In particular, we obtain a complete description of the situations where the minimal null control time is invariant with respect to all the possible choices of internal coupling matrices.
The proof relies on the notion of equivalent systems, in particular the backstepping method, a canonical $LU$-decomposition for boundary coupling matrices and a compactness-uniqueness method adapted to the null controllability property.
\end{abstract}


\vspace{0.2cm}
\MSC{35L40, 93B05}


\section{Introduction and main result}

\subsection{Problem description}

In this article we are interested in the null controllability properties of the following class of one-dimensional first-order linear hyperbolic systems, which appears for instance in linearized Saint-Venant equations and many other physical models of balance laws (see e.g. \cite[Chapter 1]{BC16} and many references therein):

\begin{equation}\label{syst}
\begin{dcases}
\pt{y}(t,x)+\Lambda(x) \px{y}(t,x)=M(x) y(t,x), \\
y_-(t,1)=u(t), \quad y_+(t,0)=Qy_-(t,0),  \\
y(0,x)=y^0(x).
\end{dcases}
\end{equation}

In \eqref{syst}, $t>0$, $x \in (0,1)$, $y(t,\cdot)$ is the state at time $t$, $y^0$ is the initial data and $u(t)$ is the control at time $t$.
We denote by $n\geq 2$ the total number of equations of the system.
The matrix $\Lambda \in C^{0,1}([0,1])^{n \times n}$ is assumed to be diagonal:
\begin{equation}\label{Lambda diag}
\Lambda =\diag(\lambda_1,\ldots,\lambda_n),
\end{equation}
with $m \geq 1$ negative speeds and $p \geq 1$ positive speeds ($m+p=n$)  such that:
\begin{equation}\label{hyp speeds}
\lambda_1(x)<\cdots<\lambda_m(x) <0<\lambda_{m+1}(x)<\cdots<\lambda_{m+p}(x), \quad \forall x \in [0,1].
\end{equation}
Finally, the matrix $M \in L^{\infty}(0,1)^{n \times n}$ couples the equations of the system inside the domain and the constant matrix $Q \in \R^{p \times m}$ couples the equations of the system on the boundary $x=0$.

All along this paper, for a vector (or vector-valued function) $v \in \R^n$ and a matrix (or matrix-valued function) $A \in \R^{n \times n}$, we use the notation
$$
v=\begin{pmatrix} v_- \\ v_+ \end{pmatrix},
\quad
A=\begin{pmatrix} A_{--} & A_{-+} \\ A_{+-} & A_{++} \end{pmatrix},
$$
where $v_- \in \R^m$, $v_+ \in \R^{p}$ and $A_{--} \in \R^{m \times m}$, $A_{-+} \in \R^{m \times p}$, $A_{+-} \in \R^{p \times m},$ $A_{++} \in \R^{p \times p}$.

We recall that the system \eqref{syst} is well posed in $(0,T)$ for every $T>0$: for every $y^0 \in L^2(0,1)^n$ and $u \in L^2(0,T)^m$, there exists a unique solution
$$y \in C^0([0,T];L^2(0,1)^n) \cap C^0([0,1];L^2(0,T)^n)$$
to the system \eqref{syst}.
By solution we mean ``solution along the characteristics'', this will be detailed in Section \ref{sect sol char} below.

The regularity $C^0([0,T];L^2(0,1)^n)$ of the solution allows us to consider control problems in the space $L^2(0,1)^n$:

\begin{definition}
Let $T>0$.
We say that the system \eqref{syst} is:
\begin{itemize}
\item
\textbf{exactly controllable in time $T$} if, for every $y^0,y^1 \in L^2(0,1)^n$, there exists $u \in L^2(0,T)^m$ such that the corresponding solution $y$ to the system \eqref{syst} in $(0,T)$ satisfies $y(T,\cdot)=y^1$.

\item
\textbf{null controllable in time $T$} if the previous property holds at least for $y^1=0$.
\end{itemize}
\end{definition}

Clearly, exact controllability implies null controllability, but the converse is not true in general.

These notions also depend on the time $T$ and, since controllability in time $T_1$ implies controllability in time $T_2$ for every $T_2 \geq T_1$, it is natural to try to find the smallest possible control time, the so-called ``minimal control time''.
This problem was recently completely solved in \cite{HO21} for the notion of exact controllability and we will investigate here what happens for the null controllability.

\begin{definition}
For any $\Lambda,M$ and $Q$ as above, we denote by $\Tinf\syst{M}{} \in [0,+\infty]$ the minimal null control time of the system \eqref{syst}, that is
$$
\Tinf\syst{M}{}=
\inf\ens{T>0 \st \text{the system \eqref{syst} is null controllable in time $T$}}.
$$
\end{definition}

The time $\Tinf\syst{M}{}$ is named ``minimal'' null control time according to the current literature, despite it is not always a minimal element of the set.
We keep this naming here, but we use the notation with the ``inf'' to avoid eventual confusions.
The time $\Tinf\syst{M}{} \in [0,+\infty]$ is thus the unique time that satisfies the following two properties:
\begin{itemize}
\item
If $T>\Tinf\syst{M}{}$, then the system \eqref{syst} is null controllable in time $T$.
\item
If $T<\Tinf\syst{M}{}$, then the system \eqref{syst} is not null controllable in time $T$.
\end{itemize}

Finally, let us introduce the elementary times $T_1(\Lambda), \ldots,T_n(\Lambda)$ defined by
\begin{equation}\label{comp Ti}
T_i(\Lambda)=
\begin{dcases}
\int_0^1 \frac{1}{-\lambda_i(\xi)} \, d\xi & \text{ if } i \in \ens{1,\ldots,m}, \\
\int_0^1 \frac{1}{\lambda_i(\xi)} \, d\xi & \text{ if } i \in \ens{m+1,\ldots,n}.
\end{dcases}
\end{equation}
For the rest of this article it is important to keep in mind that the assumption \eqref{hyp speeds} implies the following order relation among $T_i(\Lambda)$:
\begin{equation}\label{order times}
\begin{dcases}
T_1(\Lambda) \leq \cdots \leq T_m(\Lambda), \\
T_{m+p}(\Lambda) \leq \cdots \leq T_{m+1}(\Lambda).
\end{dcases}
\end{equation}

\subsection{The $LCU$ decomposition}

An important feature of the present article is that no assumption will be required on the boundary coupling matrices $Q$.
To be able to handle such a general case and state our main result we introduce a notion of canonical form.

\begin{definition}\label{def canon}
We say that a matrix $Q^0 \in \R^{p \times m}$ is in canonical form if either $Q^0=0$ or there exist an integer $\rho \geq 1$, row indices $r_1,\ldots,r_{\rho} \in \ens{1,\ldots,p}$ with $r_1<\cdots<r_{\rho}$ and distinct column indices $c_1, \ldots, c_{\rho} \in \ens{1,\ldots,m}$ such that
$$
\begin{dcases}
q^0_{ij}=1 & \text{ if } (i,j) \in \ens{(r_1,c_1), \ldots, (r_{\rho},c_{\rho})},
\\
q^0_{ij}=0 & \text{ otherwise. }
\end{dcases}
$$
For $Q^0=0$ we set $\rho=0$ for convenience.
\end{definition}

Note that we necessarily have $\rho=\rank Q^0$.

\begin{example}\label{ex Qzero}
The matrices
$$
Q_1^0=
\begin{pmatrix}
0 & \fbox{$1$} & 0 \\
0 & 0 & \fbox{$1$} \\
0 & 0 & 0 \\
\fbox{$1$} & 0 & 0
\end{pmatrix},
\quad
Q_2^0=
\begin{pmatrix}
\fbox{$1$} & 0 & 0 & 0 \\
0 & \fbox{$1$} & 0 & 0 \\
0 & 0 & 0 & 0 \\
0 & 0 & \fbox{$1$} & 0
\end{pmatrix}.
$$
are both in canonical form, with
$$
\begin{array}{c}
\text{for $Q_1^0$: } \quad
(r_1,c_1)=(1,2), \quad
(r_2,c_2)=(2,3), \quad
(r_3,c_3)=(4,1),
\\
\text{for $Q_2^0$: } \quad
(r_1,c_1)=(1,1), \quad
(r_2,c_2)=(2,2), \quad
(r_3,c_3)=(4,3).
\end{array}
$$
\end{example}

Using the Gaussian elimination we can transform any matrix into a canonical form, this is what we will call in this article the ``$LCU$ decomposition'' (for Lower--Canonical--Upper decomposition).
More precisely, we have

\begin{proposition}\label{Gauss elim}
For every $Q \in \R^{p \times m}$, there exists a unique $Q^0 \in \R^{p \times m}$ such that the following two properties hold:
\begin{enumerate}[(i)]
\item
There exists an upper triangular matrix $U \in \R^{m \times m}$ with only ones on its diagonal and there exists an invertible lower triangular matrix $L \in \R^{p \times p}$ such that
$$LQU=Q^0.$$

\item
$Q^0$ is in canonical form.
\end{enumerate}
We call $Q^0$ the canonical form of $Q$.
\end{proposition}

We mention that, because of possible zero rows or columns of $Q^0$, the matrices $L$ and $U$ are in general not unique.

With this proposition, we can extend the definition of the indices $(r_1,c_1), \ldots, (r_{\rho},c_{\rho})$ to any nonzero matrix:

\begin{definition}\label{def indices}
For any nonzero matrix $Q \in \R^{p \times m}$, we denote by $(r_1,c_1), \ldots, (r_{\rho},c_{\rho})$ the positions of the nonzero entries of its canonical form ($r_1<\cdots<r_{\rho}$).
\end{definition}

As previously mentioned, the existence of the $LCU$ decomposition is a direct consequence of the Gaussian elimination, where the matrix $U$ corresponds to rightward column substitutions, whereas the matrix $L$ corresponds to downward row substitutions and then normalization to $1$ of the remaining nonzero entries.
Let us present some examples which will make this point clearer.

\begin{example}\label{ex Q}
We illustrate how to find the decomposition of Proposition \ref{Gauss elim} in practice.
Consider
$$
Q_1=
\begin{pmatrix}
0 & 1 & 2 \\
0 & 2 & 5 \\
0 & 1 & 2 \\
4 & -4 & 4
\end{pmatrix},
\quad
Q_2=
\begin{pmatrix}
1 & 1 & -1 & 2 \\
3 & 5 & -1 & 8 \\
0 & 1 & 1 & 1 \\
-1 & 3 & 6 & 4
\end{pmatrix}.
$$
Let us deal with $Q_1$ first.
We look at the first row, we take the first nonzero entry as pivot.
We remove the entries to the right on the same row by doing the column substitution $C_3 \leftarrow  C_3 -2C_2$, which gives
\begin{equation}\label{step ck=m}
Q_1U_1=
Q_1
\begin{pmatrix}
1 & 0 & 0 \\
0 & 1 & -2 \\
0 & 0 & 1 
\end{pmatrix}
=
\begin{pmatrix}
0 & 1 & 0 \\
0 & 2 & 1 \\
0 & 1 & 0 \\
4 & -4 & 12
\end{pmatrix}
.
\end{equation}
We now look at the next row and take as new pivot the first nonzero entry that is not in the column of the previous pivot, that is, not in $C_2$.
Since there is no entry to the right of this new pivot, there is nothing to do and we move to the next row.
Since this next row has no nonzero element which is not in $C_2,C_3$, we move again to the next and last row.
We take as new pivot the first nonzero entry that is not in $C_2$ or $C_3$ and we remove the entries to the right on the same row by doing the column substitutions $C_2 \leftarrow  C_2 +C_1$ and $C_3 \leftarrow C_3-3C_1$, which gives
$$
Q_1U_1U_2=
Q_1U_1
\begin{pmatrix}
1 & 1 & -3 \\
0 & 1 & 0 \\
0 & 0 & 1 
\end{pmatrix}
=
\begin{pmatrix}
0 & 1 & 0 \\
0 & 2 & 1 \\
0 & 1 & 0 \\
4 & 0 & 0
\end{pmatrix}
.
$$
Working then on the rows with downward substitutions only (starting with the first row) and finally normalizing to $1$ the remaining nonzero entries, we see that $Q_1$ becomes $Q_1^0$ of Example \ref{ex Qzero}.
Similarly, it can be checked that the canonical form of $Q_2$ is in fact $Q_2^0$ of Example \ref{ex Qzero}.
\end{example}

\begin{remark}
Observe that we only need to compute the matrix $U$ in order to find the indices $(r_1,c_1), \ldots, (r_{\rho},c_{\rho})$.
\end{remark}

The uniqueness of the $LCU$ decomposition is less straightforward and we refer for instance to the arguments in the proof of \cite[Theorem 1]{DJM06} or to \cite[Appendix A]{HO21} for a proof.

\begin{remark}
In the Gaussian elimination process described above, we absolutely do not want to perform any permutation of the rows.
This is because we have ordered the speeds of our system in a particular way (recall \eqref{hyp speeds}).
The fact that we use right multiplication by upper triangular matrices and left multiplication by lower triangular matrices is also dictated by this choice of order (for instance in \cite{HO21} the speeds were ordered differently and right multiplication by lower triangular matrices was considered instead).
\end{remark}

\subsection{Literature}\label{sect lit}

Boundary controllability of one-dimensional first-order hyperbolic systems has been widely investigated since the late 1960s.
Two pioneering works are \cite{Rus67} and the celebrated survey paper \cite{Rus78}, in which the author established the null controllability of the system \eqref{syst} in any time $T \geq \TRus$, where $\TRus$ is the sum of the largest transport times from opposite directions, that is,
$$\TRus=T_{m+1}(\Lambda)+T_m(\Lambda).$$
An important feature of this result is that it is valid whatever are the internal and boundary coupling matrices $M$ and $Q$.
In other words, the time $\TRus$ gives an upper bound for the minimal null control time $\Tinf(\Lambda,M,Q)$ with respect to these matrices.

In general, no better time can be expected.
More precisely, it is easy to see that there exist matrices $M$ and $Q$ such that $\Tinf(\Lambda,M,Q)=\TRus$ (simply take $M=0$ and $Q$ the matrix whose entries are all equal to zero except for $q_{1,m}=1$).
However, for most of the matrices $M$ and $Q$, this upper bound is too large.
Indeed, by just slightly restricting the class of such matrices (in particular, for $Q$), it is possible to have a strictly better upper bound than $\TRus$.

This fact was already observed in \cite{Rus78}, where the author tried to find the minimal null control time in the particular case of conservation laws ($M=0$), rightly by looking more closely at the properties of the boundary coupling matrix.
He could not solve this problem though and he left it as an open problem (\cite[Remark p. 656]{Rus78}).
This was eventually solved few years later in \cite{Wec82}, where the author gave an explicit expression of the minimal null control time in terms of some indices related to $Q$.

Concerning systems of balance laws ($M \neq 0$), finding the minimal null control time for arbitrary $M$ and $Q$ is still an open challenging problem.
Recently, there has been a resurgence on the characterization of such a time.
A first result in this direction has been obtained in \cite{CN19} with an improvement of the upper bound $\TRus$ for a certain class of boundary coupling matrices $Q$.
More precisely, they considered the class $\mathcal{B}$ defined by
\begin{equation}\label{def setCN}
\setCN=
\ens{Q \in \R^{p \times m} \st \text{\eqref{cond CN19} is satisfied for every $i \in \ens{1,\ldots,\min\ens{p,m-1}}$}},
\end{equation}
where the condition \eqref{cond CN19} is:
\begin{equation}\label{cond CN19}
\text{the $i \times i$ matrix formed from the first $i$ rows and the first $i$ columns of $Q$ is invertible,}
\end{equation}
(it is understood that the set $\setCN$ is empty when $m=1$).
For this class of boundary coupling matrices, the authors then showed that the upper bound $\TRus$ can be reduced to the time $\TCN$ defined by
\begin{equation}\label{def TCN}
\TCN=
\begin{dcases}
\max \ens{ \max_{k \in \ens{1,\ldots,p}} T_{m+k}(\Lambda)+T_k(\Lambda), \quad T_{m}(\Lambda)} & \text{ if } m \geq p,
\\
\max_{k \in \ens{1,\ldots,m}} T_{m+k}(\Lambda)+T_k(\Lambda) & \text{ if } m<p.
\end{dcases}
\end{equation}
This was first done for some generic internal coupling matrices or under rather stringent conditions (\cite[Theorem 1.1 and 1.5]{CN19}) but the same authors were then able to remove these restrictions in \cite[Theorem 1 and 3]{CN21}.

On the other hand, when the boundary coupling matrix $Q$ is full row rank, the problem of finding the minimal null control time, and not only an upper bound, has also been recently completely solved in \cite{HO21}.
More precisely, it is proved in \cite[Theorem 1.12 and Remark 1.3]{HO21} that
$$
\rank Q=p
\quad \Longrightarrow \quad
\Tinf(\Lambda,M,Q)=
\max\ens{
\max_{k \in \ens{1,\ldots,p}}
T_{m+k}(\Lambda)+T_{c_k}(\Lambda),
\quad
T_m(\Lambda)
}.
$$
We see in this case that the minimal null control time has the remarkable property to be independent of the internal coupling matrix $M$.
In particular, this is the same time as the one found for conservation laws in \cite{Wec82}, yet with a more explicit expression.
For $m>p$, this generalizes the aforementioned results of \cite{CN19,CN21} in two ways: firstly, this is a result for arbitrary full row rank boundary coupling matrices (not only for $Q \in \setCN$) and, secondly, this obviously establishes that no better time can be obtained (even for $Q \in \setCN$ this is not proved in \cite{CN19,CN21}).
We mention this because the results of the present paper will share these two features.

For the special case of $2 \times 2$ systems, the minimal null control time has also been found in \cite{CVKB13} when the boundary coupling matrix (which is then a scalar) is not zero and in \cite{HO20-pre} when the boundary coupling is reduced to zero.
Notably, in the second situation, the minimal null control time depends on the behavior of the internal coupling matrix $M$ (\cite[Theorem 1.5]{HO20-pre}).

Finally, we would like to mention the related works \cite{CHOS21,CN21-pre,AKM21-pre} concerning time-dependent systems and \cite{Li10,LR10,Hu15,CN20-pre1, CN20-pre2} for quasilinear systems.

%
%
%
%
%
%
%
%

As we have discussed, finding what exactly is the minimal null control time turns out to be a difficult task.
Instead, in this article we propose to look for the smallest and largest values that the minimal null control time $\Tinf(\Lambda,M,Q)$ can take with respect to the internal coupling matrix $M$.
Our main result is an explicit and easy-to-compute formula for both of these times.
We will also completely characterize all the parameters $\Lambda$ and $Q$ for which $\Tinf(\Lambda,M,Q)$ is invariant with respect to all $M \in L^{\infty}(0,1)^{n \times n}$.
We will show that our results generalize all the known works that have been previously quoted.
In the course of the proof we will obtain some new results even for conservation laws ($M=0$), notably with an explicit feedback law stabilizing the system in the minimal time.

Our proof relies on the notion of equivalent systems, in particular the backstepping method with the results of \cite{HDMVK16, HVDMK19}, the introduction of a canonical $LU$-decomposition for boundary coupling matrix $Q$ in the same spirit as in \cite{HO21}, as well as a compactness-uniqueness method adapted to the null controllability inspired from the works \cite{CN21,DO18}.

\subsection{Main result and comments}

As we have seen in the previous section, to explicitly characterize $\Tinf(\Lambda,M,Q)$ for arbitrary $M$ and $Q$ is still a challenging open problem.
Instead, we propose to find the smallest and largest values that it can take with respect to the internal coupling matrix $M$.

\begin{definition}\label{def Topt Tunif}
We define
\begin{gather*}
\infT(\Lambda,Q)=\inf \ens{\Tinf\syst{M}{} \st M \in L^{\infty}(0,1)^{n \times n}}, \\
\supT(\Lambda,Q)=\sup \ens{\Tinf\syst{M}{} \st M \in L^{\infty}(0,1)^{n \times n}}.
\end{gather*}
\end{definition}

The main result of the present paper is the following explicit characterization of these two quantities:

\begin{theorem}\label{main thm}
Let $\Lambda \in C^{0,1}([0,1])^{n \times n}$ satisfy \eqref{Lambda diag}-\eqref{hyp speeds} and let $Q \in \R^{p \times m}$ be fixed.
\begin{enumerate}[(i)]
\item\label{item infT}
We have
\begin{equation}\label{caract infT}
\infT(\Lambda,Q)=
\max\ens{
\max_{k \in \ens{1,\ldots,\rho}}
T_{m+r_k}(\Lambda)+T_{c_k}(\Lambda), \quad
T_{m+1}(\Lambda)
, \quad
T_m(\Lambda)
},
\end{equation}
where we recall that the indices $(r_k,c_k)$ are defined in Definition \ref{def indices}.

\item\label{item supT}
We have
\begin{equation}\label{caract supT}
\supT(\Lambda,Q)=
\max\ens{
\max_{k \in \ens{1,\ldots,\rho_0}}
T_{m+k}(\Lambda)+T_{c_k}(\Lambda),
\quad
T_{m+\rho_0+1}(\Lambda)+T_m(\Lambda)
},
\end{equation}
where $\rho_0$ is the largest integer $i \in \ens{1,\ldots,p}$ such that
$$
\text{the $i \times m$ matrix formed from the first $i$ rows of $Q$ has rank $i$,}
$$
with $\rho_0=0$ if the first row of $Q$ is equal to zero.
\end{enumerate}

\end{theorem}

In the statement of Theorem \ref{main thm}, we use the convention that the undefined quantities are simply not taken into account, which more precisely gives:
\begin{itemize}
\item
If $\rho=0$, then $\infT(\Lambda,0)=\max\ens{T_{m+1}(\Lambda), \quad T_m(\Lambda)}$.

\item
If $\rho_0=0$, then $\supT(\Lambda,Q)=T_{m+1}(\Lambda)+T_m(\Lambda)$.

\item
If $\rho_0=p$, then $\supT(\Lambda,Q)=\max\ens{\max_{k \in \ens{1,\ldots,p}} T_{m+k}(\Lambda)+T_{c_k}(\Lambda), \quad T_m(\Lambda)}$.
\end{itemize}

An equivalent definition of $\rho_0$ is (when the first row of $Q$ is not equal to zero)
\begin{equation}\label{r less than rhozero}
\rho_0=\max\ens{i \in \ens{1,\ldots,p} \st r_k=k, \quad \forall k \in \ens{1,\ldots,i}}.
\end{equation}
We emphasize that $\rho_0$ is defined for any $Q$, it is not a condition like $Q \in \setCN$.

By investigating the possibilities of equality $\infT(\Lambda,Q)=\supT(\Lambda,Q)$, we obtain the following important consequence of Theorem \ref{main thm}:

\begin{corollary}\label{cor main thm}
Let $\Lambda \in C^{0,1}([0,1])^{n \times n}$ satisfy \eqref{Lambda diag}-\eqref{hyp speeds} and let $Q \in \R^{p \times m}$ be fixed.
The map $M \mapsto \Tinf(\Lambda,M,Q)$ is constant on $L^{\infty}(0,1)^{n \times n}$ if, and only if, $\Lambda$ and $Q$ satisfy
\begin{equation}\label{CNS Tinf indep M}
\rho_0=p \quad \text{ or } \quad \left(0<\rho_0<p \quad \text{ and } \quad \max_{k \in \ens{1,\ldots,\rho_0}}
T_{m+k}(\Lambda)+T_{c_k}(\Lambda)
\geq
T_{m+\rho_0+1}(\Lambda)+T_m(\Lambda)
\right).
\end{equation}
\end{corollary}

\begin{remark}\label{rem times are reached}
In the proof of Theorem \ref{main thm}, we will show in fact that the infimum in \eqref{caract infT} and the supremum in \eqref{caract supT} are reached for some special matrices $M$.
More precisely, we will show that:
\begin{itemize}
\item
The infimum in \eqref{caract infT} is reached for $M=0$.

\item
The supremum in \eqref{caract supT} is reached for $M=0$ if the condition \eqref{CNS Tinf indep M} holds.

\item
If the condition \eqref{CNS Tinf indep M} fails, then the supremum in \eqref{caract supT} is reached for the matrix $M$ whose entries are all equal to zero, except for
$$
m_{m+i, m}(x)=\frac{\lambda_{m+i}(x)-\lambda_m(x)}{-\lambda_m(x)}\ell^{i,\rho_0+1}, \quad \forall i \in \ens{\rho_0+1,\ldots,p},
$$
where $L^{-1}=(\ell^{ij})_{1 \leq i,j \leq p}$ and $L$ is any matrix $L$ coming from the $LCU$ decomposition of $Q$.
\end{itemize}

\end{remark}

\begin{example}\label{ex compt times}
For $Q_1 \in \R^{4 \times 3}$ and $Q_2 \in \R^{4 \times 4}$ of Example \ref{ex Q} we have (recall \eqref{order times})
$$
\begin{array}{rl}
\infT(\Lambda,Q_1) &=
\max\ens{
T_4(\Lambda)+T_2(\Lambda), \quad T_5(\Lambda)+T_3(\Lambda), \quad T_7(\Lambda)+T_1(\Lambda), \quad
T_4(\Lambda)
, \quad
T_3(\Lambda)
}
\\
&=\max\ens{
T_4(\Lambda)+T_2(\Lambda), \quad T_5(\Lambda)+T_3(\Lambda)
},
\end{array}
$$

$$
\begin{array}{rl}
\supT(\Lambda,Q_1) &=
\max\ens{
T_4(\Lambda)+T_2(\Lambda), \quad T_5(\Lambda)+T_3(\Lambda), \quad T_6(\Lambda)+T_3(\Lambda)
}
\\
&=\max\ens{
T_4(\Lambda)+T_2(\Lambda), \quad T_5(\Lambda)+T_3(\Lambda)
}
\\
&=\infT(\Lambda,Q_1),
\end{array}
$$
and
$$
\begin{array}{rl}
\infT(\Lambda,Q_2) &=
\max\ens{
T_5(\Lambda)+T_1(\Lambda), \quad T_6(\Lambda)+T_2(\Lambda), \quad T_8(\Lambda)+T_3(\Lambda), \quad
T_5(\Lambda)
, \quad
T_4(\Lambda)
}
\\
&=\max\ens{
T_5(\Lambda)+T_1(\Lambda), \quad T_6(\Lambda)+T_2(\Lambda), \quad T_8(\Lambda)+T_3(\Lambda), \quad
T_4(\Lambda)
},
\end{array}
$$

$$
\supT(\Lambda,Q_2) =
\max\ens{
T_5(\Lambda)+T_1(\Lambda), \quad T_6(\Lambda)+T_2(\Lambda), \quad T_7(\Lambda)+T_4(\Lambda)
}.
$$

\end{example}

\begin{remark}
If during the computations of the indices $(r_k,c_k)$ we arrive at the last column, that is if we have
\begin{equation}\label{ck=m}
c_{k_0}=m,
\end{equation}
for some $k_0 \in \ens{1,\ldots,p}$, then there is no need to find the next indices to be able to compute $\infT(\Lambda,Q)$ and $\supT(\Lambda,Q)$ since we know that the corresponding times will not be taken into account (because of \eqref{order times}).
For instance, for the matrix $Q_1$ of Example \ref{ex Q} we can stop after the very first step \eqref{step ck=m} since it gives $c_2=3$, there is no need to go on and compute $U_2$.
\end{remark}

\begin{remark}\label{rem lit}
Theorem \ref{main thm} and its corollary generalize all the results of the literature that we are aware of on the null controllability of systems of the form \eqref{syst} (except for the special case $n=2$, which has been completely solved in \cite{CVKB13,HO20-pre}):

\begin{itemize}
\item
When the matrix $Q$ is full row rank, that is,
\begin{equation}\label{Q full row rank}
\rank Q=p,
\end{equation}
exact and null controllability are equivalent properties for the system \eqref{syst} (see e.g. \cite[Remark 1.3]{HO21}) and it has been shown in \cite[Theorem 4.1]{HO21} that $\Tinf(\Lambda,M,Q)$ is independent of $M$ in that situation.
Under the rank condition \eqref{Q full row rank}, it is clear that $\rho_0=p$ and the condition \eqref{CNS Tinf indep M} is thus satisfied.
It then follows from Corollary \ref{cor main thm} that $\Tinf(\Lambda,M,Q)$ is independent of $M$.
Therefore, our result encompasses the one of \cite{HO21}.

\item
When $m \leq p$ and $Q \in \setCN$ (defined in \eqref{def setCN}), it has been established in \cite[Theorem 1]{CN21} that
$$
\supT(\Lambda,Q)
\leq
\TCN,
$$
where we recall that $\TCN$ is given by \eqref{def TCN}.
In that case, we see that $r_k=c_k=k$ for every $k \in \ens{1,\ldots,m-1}$ and either $\rho_0=m-1$ or $\rho_0=m$.
In all cases, we can check that
$$
\max\ens{
\max_{k \in \ens{1,\ldots,\rho_0}}
T_{m+k}(\Lambda)+T_{c_k}(\Lambda),
\quad
T_{m+\rho_0+1}(\Lambda)+T_m(\Lambda)
}
=\TCN.
$$
Therefore, item \ref{item supT} of Theorem \ref{main thm} generalizes \cite[Theorem 1]{CN21}, which corresponded only to the inequality ``$\leq$'' and only valid for matrices $Q \in \setCN$, but excluded for instance the matrices presented in Example \ref{ex Q}.
We mention that, since the speeds are ordered, we cannot simply renumber the unknowns so that, after this transformation, the new matrix $Q$ belongs to $\setCN$.

\item
In fact, when $\rho_0=m$ in the previous point, the minimal null control time does not depend on $M$.
More generally, if the condition \eqref{ck=m} holds for some $k_0 \leq \rho_0$, then the condition \eqref{CNS Tinf indep M} is satisfied (because of \eqref{order times}) and it follows from Corollary \ref{cor main thm} that
$$
\Tinf(\Lambda,M,Q)=
\max_{k \in \ens{1,\ldots,k_0}}
T_{m+k}(\Lambda)+T_{c_k}(\Lambda).
$$
For instance, this condition is satisfied when the matrix $Q$ has the block decomposition
$$
Q=\begin{pmatrix}
Q'
\\
Q''
\end{pmatrix},
\quad \rank Q'=m,
$$
where $Q' \in \R^{m \times m}$ and $Q'' \in \R^{(p-m) \times m}$.
\end{itemize}

\end{remark}


\subsection{Equivalent systems}

The proof of our main result will first consist in transforming our initial system \eqref{syst} into ``equivalent'' systems (from a controllability point of view) which have a simpler coupling structure.
Let us make this notion of equivalent systems precise here.
We will introduce it for a slightly broader class of systems than \eqref{syst} because of the nature of the transformations that we will use in the sequel, this will be clear from Section \ref{sect backstepping}.
All the systems of this paper will have the following form:
\begin{equation}\label{syst gen}
\begin{dcases}
\pt{y}(t,x)+\Lambda(x) \px{y}(t,x)=M(x) y(t,x)+G(x)y_-(t,0), \\
y_-(t,1)=u(t), \quad y_+(t,0)=Qy_-(t,0),  \\
y(0,x)=y^0(x),
\end{dcases}
\end{equation}
where $M \in L^{\infty}(0,1)^{n \times n}$ and $Q \in\R^{p \times m}$ as before, and $G \in L^{\infty}(0,1)^{n \times m}$.
Therefore, \eqref{syst gen} is similar to \eqref{syst} but it has the extra term with $G$.
This system is well posed and the notions of controllability are similarly defined (see Section \ref{sect sol char} below).

In what follows, we will refer to a system of the general form \eqref{syst gen} as
$$(\Lambda,M,Q,G).$$
When a system does not contain a parameter ($M$ or $G$) we will use the notation $-$ rather than writing $0$, for instance we will use $(\Lambda,M,Q,-)$ when the system does not contain $G$.
The minimal null control time of the system $(\Lambda,M,Q,G)$ will be denoted by $\Tinf(\Lambda,M,Q,G)$ (for consistency, we will keep using the notation $\Tinf(\Lambda,M,Q)$ rather than $\Tinf(\Lambda,M,Q,-)$).

Let us now give the precise definition of what we mean by equivalent systems in this work:

\begin{definition}\label{def syst equiv}
We say that two systems $(\Lambda,M_1,Q_1,G_1)$ and $(\Lambda,M_2,Q_2,G_2)$ are equivalent, and we write
$$(\Lambda,M_1,Q_1,G_1) \sim (\Lambda,M_2,Q_2,G_2),$$
if there exists an invertible bounded linear transformation
$$L:L^2(0,1)^n \longrightarrow L^2(0,1)^n,$$
such that, for every $T>0$, the induced map $\tilde{L}:C^0([0,T];L^2(0,1)^n) \longrightarrow C^0([0,T];L^2(0,1)^n)$ defined by $(\tilde{L}y)(t)=L(y(t))$ for every $t \in [0,T]$ satisfies
$$\tilde{L}(S_1)=S_2,$$
where $S_i$ ($i=1,2$) denotes the space of all the solutions $y$ to the system $(\Lambda,M_i,Q_i,G_i)$ in $(0,T)$.
\end{definition}

It is not difficult to check that $\sim$ is an equivalence relation and that two equivalent systems share the same controllability properties:

\begin{proposition}\label{basic prop equiv systs}
Let $(\Lambda,M_1,Q_1,G_1) \sim (\Lambda,M_2,Q_2,G_2)$ be two equivalent systems.
Then, for every $T>0$, the system $(\Lambda,M_1,Q_1,G_1)$ is null controllable in time $T$ if, and only if, the system $(\Lambda,M_2,Q_2,G_2)$ is null controllable in time $T$.
\end{proposition}

In particular, two equivalent systems have the same minimal null control time.
However, the converse is not true in general, an example has been detailed in Appendix \ref{sect counterexample}.

\begin{remark}
Let us emphasize that the notion of equivalent systems that we introduced here does not care how the control from one system is obtained from the control of the other system.
It is different from the notion of (feedback) equivalence introduced in the seminal work \cite{Bru70} in finite dimension, which was designed to transfer the stabilization properties of one system to another and thus required a more specific link between the two systems.
\end{remark}

\subsection{Outline of the proof}

Since the proof of our main result involves many transformations, let us give a quick overview of the main steps before going into detail:

\begin{enumerate}[1)]
\item
First of all, we show in Section \ref{sect backstepping} that 
$$
(\Lambda,M,Q,-)
\equi
(\Lambda,-,Q,G),
$$
for some $G$.
It is nothing but a fundamental result of \cite{HVDMK19,HDMVK16} that we rephrase here with the notion of equivalent systems.
Consequently, we only have to focus on systems of the form $(\Lambda,-,Q,G)$ in the sequel, which have the advantage of having a simpler coupling structure.

\item
In Section \ref{sect boundary coupling matrix}, we show that the boundary coupling matrix $Q$ can always be assumed in canonical form (Definition \ref{def canon}):
$$
(\Lambda,-,Q,G)
\equi
(\Lambda,-,Q^0,\widetilde{G}),
$$
for some $\widetilde{G}$.
This is an important step that greatly simplifies the coupling structure of the system.

\item
Notably, this allows us to characterize in Section \ref{sect infT} the smallest value of the minimal null control time.
More precisely, we first establish that
$$
\inf\ens{\Tinf(\Lambda,-,Q^0,\widetilde{G}) \st \widetilde{G} \in L^{\infty}(0,1)^{n \times m}}
$$
is equal to the quantity on the right-hand side of the equality \eqref{caract infT}.
This is done by using a similar argument to the one in \cite{HO20-pre}.
We then show how to deduce the corresponding result for the initial system $(\Lambda,M,Q,-)$, thus proving the first part of our main result.

\item
In view of the proof of the second part of our main result, we first show in Section \ref{sect reduc to casc} how to use the canonical form of $Q^0$ to prove that
$$
(\Lambda,-,Q^0,\widetilde{G})
\equi
\left(\Lambda,-,Q^0,
\begin{pmatrix}
\widetilde{G}_{--} \\
\widehat{G}_{+-}
\end{pmatrix}
\right),
$$
for some $\widehat{G}_{+-}$ which has the following structure:
\begin{equation}\label{cond Gtilde}
\hat{g}_{m+i, c_k}=0, \quad \forall k \in \ens{1,\ldots,\rho}, \, \forall i \geq r_k.
\end{equation}

\item
In Section \ref{sect compactness-uniqueness} we then prove that the coupling term $\widetilde{G}_{--}$ has no influence on the minimal null control time:
$$
\Tinf\left(\Lambda,-,Q^0,
\begin{pmatrix}
\widetilde{G}_{--} \\
\widehat{G}_{+-}
\end{pmatrix}
\right)
=
\Tinf\left(\Lambda,-,Q^0,
\begin{pmatrix}
0 \\
\widehat{G}_{+-}
\end{pmatrix}
\right).
$$
Unlike all the other steps, the proof is not based on the construction of a suitable transformation, it is based on a general compactness-uniqueness method adapted to the null controllability property and inspired from the previous works \cite{CN21,DO18}.

\item
Finally, in Section \ref{sect caract supT}, we characterize the largest value of the minimal null control time.
More precisely, we first show that
$$
\sup\ens{\Tinf\left(\Lambda,-,Q^0,
\begin{pmatrix}
0 \\
\widehat{G}_{+-}
\end{pmatrix}
\right) \st \text{$\widehat{G}_{+-}$ satisfies \eqref{cond Gtilde}}}
$$
is equal to the quantity on the right-hand side of the equality \eqref{caract supT}.
We then show how to deduce the corresponding result for the initial system $(\Lambda,M,Q,-)$, thus proving the second part of our main result.

\end{enumerate}

\begin{remark}
All the steps described above are constructive, except for the one invoking a compactness-uniqueness argument.
It would be interesting to be able to replace this step by a constructive approach (if possible).
\end{remark}

\section{Notations and solution along the characteristics}\label{sect sol char}

Before proceeding to the proof of our main result, we introduce in this section some notations and recall some results concerning the well-posedness of the non standard systems of the form \eqref{syst gen}.

\subsection{The characteristics}\label{sect caract}

We start with the characteristic curves associated with the system \eqref{syst gen}.

\begin{itemize}
\item
First of all, throughout this paper it is convenient to extend $\lambda_1,\ldots,\lambda_n$ to functions of $\R$ (still denoted by the same) such that $\lambda_1,\ldots,\lambda_n \in C^{0,1}(\R)$ and 
\begin{equation}\label{hyp speeds bis}
\lambda_1(x)<\cdots<\lambda_m(x) \leq -\epsilon <0< \epsilon \leq \lambda_{m+1}(x)<\cdots<\lambda_{m+p}(x), \quad \forall x \in \R,
\end{equation}
for some $\epsilon>0$ small enough.
Since all the results of the present paper depend only on the values of $\lambda_1,\ldots,\lambda_n$ in $[0,1]$, they do not depend on such an extension.
\end{itemize}

In what follows, $i \in \ens{1,\ldots,n}$ is fixed.

\begin{itemize}
\item
Let $\chi_i$ be the flow associated with $\lambda_i$, i.e. for every $(t,x) \in \R \times \R$, the function $s \mapsto \chi_i(s;t,x)$ is the solution to the ordinary differential equation (ODE)
\begin{equation}\label{ODE xi}
\begin{dcases}
\pas{\chi_i}(s;t,x)=\lambda_i(\chi_i(s;t,x)), \quad \forall s \in \R, \\
\chi_i(t;t,x)=x.
\end{dcases}
\end{equation}
The existence and uniqueness of a (global) solution to the ODE \eqref{ODE xi} follows from the (global) Cauchy-Lipschitz theorem (see e.g. \cite[Theorem II.1.1]{Har02}).
The uniqueness also yields the important group property
\begin{equation}\label{chi invariance}
\chi_i\left(\sigma;s,\chi_i(s;t,x)\right)=\chi_i(\sigma;t,x), \quad \forall \sigma,s \in \R.
\end{equation}

\item
Let us now introduce the entry and exit times $\ssin_i(t,x),\ssout_i(t,x) \in \R$ of the flow $\chi_i(\cdot;t,x)$ inside the domain $[0,1]$, i.e. the respective unique solutions to
$$
\begin{dcases}
\chi_i(\ssin_i(t,x);t,x)=1, \quad \chi_i(\ssout_i(t,x);t,x)=0, & \text{ if } i \in \ens{1,\ldots,m}, \\
\chi_i(\ssin_i(t,x);t,x)=0, \quad \chi_i(\ssout_i(t,x);t,x)=1, & \text{ if } i \in \ens{m+1,\ldots,n}.
\end{dcases}
$$
Their existence and uniqueness are guaranteed by the condition \eqref{hyp speeds bis}.

\item
Since $\lambda_i$ does not depend on time, we have an explicit formula for the inverse function $\theta \mapsto \chi_i^{-1}(\theta;t,x)$.
Indeed, it solves
$$
\begin{dcases}
\frac{\partial (\chi_i^{-1})}{\partial \theta}(\theta;t,x)=\frac{1}{\pas{\chi_i}\left(\chi_i^{-1}(\theta;t,x);t,x\right)}=\frac{1}{\lambda_i(\theta)}, \quad \forall \theta \in \R,\\
\chi_i^{-1}(x;t,x)=t,
\end{dcases}
$$
which gives
$$
\chi_i^{-1}(\theta;t,x)=t+\int_x^\theta \frac{1}{\lambda_i(\xi)} \, d\xi.
$$
It follows that
\begin{equation}\label{explicit formula ssin}
\begin{dcases}
\ssin_i(t,x)=t-\int_x^1 \frac{1}{-\lambda_i(\xi)} \, d\xi, \quad \ssout_i(t,x)=t+\int_0^x \frac{1}{-\lambda_i(\xi)} \, d\xi, & \text{ if } i \in \ens{1,\ldots,m}, \\
\ssin_i(t,x)=t-\int_0^x \frac{1}{\lambda_i(\xi)} \, d\xi, \quad \ssout_i(t,x)=t+\int_x^1 \frac{1}{\lambda_i(\xi)} \, d\xi, & \text{ if } i \in \ens{m+1,\ldots,n}.
\end{dcases}
\end{equation}

\item
We have the following monotonic properties:
\begin{equation}\label{ssin monotonicity}
\begin{dcases}
\pt{\ssin_i}>0, \quad \px{\ssin_i}>0, \quad \pt{\ssout_i}>0, \quad \px{\ssout_i}>0, & \text{ if } i \in \ens{1,\ldots,m}, \\
\pt{\ssin_i}>0, \quad \px{\ssin_i}<0, \quad \pt{\ssout_i}>0, \quad \px{\ssout_i}<0, & \text{ if } i \in \ens{m+1,\ldots,n},
\end{dcases}
\end{equation}
and the following inverse formula, valid for every $s,t \in \R$:
\begin{equation}\label{inv sin}
\begin{dcases}
s<\ssout_i(t,1) \quad \Longleftrightarrow \quad \ssin_i(s,0)<t, & \text{ if } i \in \ens{1,\ldots,m}, \\
s<\ssout_i(t,0) \quad \Longleftrightarrow \quad \ssin_i(s,1)<t, & \text{ if } i \in \ens{m+1,\ldots,n}.
\end{dcases}
\end{equation}

\item
Note as well that (recall \eqref{comp Ti})
\begin{equation}\label{def Ti gen}
T_i(\Lambda)=
\begin{dcases}
\ssout_i(0,1) & \text{ if } i \in \ens{1,\ldots,m}, \\
\ssout_i(0,0) & \text{ if } i \in \ens{m+1,\ldots,n}.
\end{dcases}
\end{equation}

\item
Finally, we introduce the non negative and increasing function $\phi_i \in C^{1,1}(\R)$ defined by
\begin{equation}\label{def phi}
\phi_i(x)=
\begin{dcases}
\int_0^x \frac{1}{-\lambda_i(\xi)} \, d\xi & \text{ if } i \in \ens{1,\ldots,m}, \\
\int_0^x \frac{1}{\lambda_i(\xi)} \, d\xi & \text{ if } i \in \ens{m+1,\ldots,n}.
\end{dcases}
\end{equation}
Note that it is a bijection from $[0,1]$ to $[0,T_i(\Lambda)]$.

\end{itemize}

\subsection{Solution along the characteristics}

Let us now introduce the notion of solution for systems of the form \eqref{syst gen}.
To this end, we have to restrict our discussion to the domain where the system evolves, i.e. on $(0,T)\times(0,1)$, $T>0$ being fixed.
For every $(t,x) \in (0,T)\times(0,1)$, we have
$$(s,\chi_i(s;t,x)) \in (0,t)\times(0,1), \quad \forall s \in (\ssinb_i(t,x),t),$$
where we introduced
$$\ssinb_i(t,x)=\max\ens{0,\ssin_i(t,x)}<t.$$

We now proceed to formal computations in order to introduce the notion of solution for non smooth functions $y$.
Writing the $i$-th equation of the system \eqref{syst gen} along the characteristic $\chi_i(s;t,x)$ for $s \in [\ssinb_i(t,x),t]$, and using the chain rules yields the ODE
\begin{equation}\label{sol along char}
\left\{\begin{array}{l}
\ds \dds y_i\left(s,\chi_i(s;t,x)\right)=
\sum_{j=1}^n m_{ij}\left(\chi_i(s;t,x)\right)y_j\left(s,\chi_i(s;t,x)\right)
+\sum_{j=1}^m g_{ij}\left(\chi_i(s;t,x)\right)y_j\left(s,0\right), \\
\ds y_i\left(\ssinb_i(t,x),\chi_i(\ssinb_i(t,x);t,x)\right)=b_i\left(y^0_i,u_i,y_-(\cdot,0)\right)(t,x),
 \end{array}\right.
\end{equation}
where the initial condition $b_i(y^0_i,u_i,y_-(\cdot,0))(t,x)$ is given by the appropriate boundary or initial conditions in \eqref{syst gen}:
\begin{itemize}
\item
for $i \in \ens{1,\ldots,m}$,
\begin{equation}\label{def bi i<m}
b_i\left(y^0_i,u_i,y_-(\cdot,0)\right)(t,x)=
\begin{dcases}
u_i(\ssin_i(t,x)) & \text{ if } \ssin_i(t,x)>0, \\
y^0_i(\chi_i(0;t,x)) & \text{ if } \ssin_i(t,x)<0,
\end{dcases}
\end{equation}

\item
for $i \in \ens{m+1,\ldots,n}$,
\begin{equation}\label{def bi i>m}
b_i\left(y^0_i,u_i,y_-(\cdot,0)\right)(t,x)=
\begin{dcases}
\sum_{j=1}^m q_{i-m,j} y_j(\ssin_i(t,x),0)& \text{ if } \ssin_i(t,x)>0, \\
y^0_i(\chi_i(0;t,x)) & \mbox{ if } \ssin_i(t,x)<0.
\end{dcases}
\end{equation}
\end{itemize}
Integrating the ODE \eqref{sol along char} over $s \in [\ssinb_i(t,x),t]$, we obtain the following system of integral equations:
\begin{multline}\label{sol char def 1}
y_i(t,x)=b_i\left(y^0_i,u_i,y_-(\cdot,0)\right)(t,x)+\sum_{j=1}^n \int_{\ssinb_i(t,x)}^{t} m_{ij}(\chi_i(s;t,x))y_j(s,\chi_i(s;t,x)) \, ds
\\
+\sum_{j=1}^m \int_{\ssinb_i(t,x)}^t g_{ij}\left(\chi_i(s;t,x)\right)y_j\left(s,0\right) \, ds.
\end{multline}
This leads to the following notion of solution called ``solution along the characteristics'':

\begin{definition}\label{def sol char}
Let $T>0$, $y^0 \in L^2(0,1)^n$ and $u \in L^2(0,T)^m$ be fixed.
We say that a function $y:(0,T)\times(0,1) \longrightarrow \R^n$ is a solution to the system \eqref{syst gen} in $(0,T)$ if
$$y \in C^0([0,T];L^2(0,1)^n) \cap C^0([0,1];L^2(0,T)^n),$$
and if the integral equation \eqref{sol char def 1} is satisfied for every $i \in \ens{1,\ldots,n}$ and for a.e. $(t,x) \in (0,T) \times (0,1)$.
\end{definition}

Using the Banach fixed-point theorem and suitable estimates, we can establish that the system \eqref{syst gen} is globally well posed in this sense:

\begin{theorem}
For every $T>0$, $y^0 \in L^2(0,1)^n$ and $u \in L^2(0,T)^m$, there exists a unique solution $y \in C^0([0,T];L^2(0,1)^n) \cap C^0([0,1];L^2(0,T)^n)$ to the system \eqref{syst gen} in $(0,T)$.
Moreover, we have
\begin{equation}\label{estim well posed}
\norm{y}_{C^0([0,T];L^2(0,1)^n)}
+\norm{y}_{C^0([0,1];L^2(0,T)^n)}
\leq C\left(\norm{y^0}_{L^2(0,1)^n}+\norm{u}_{L^2(0,T)^m}\right),
\end{equation}
for some $C>0$ that does not depend on $y^0$ nor $u$.
\end{theorem}

For a proof of this result, we refer for instance to \cite[Appendix A.2]{CHOS21} (see also \cite[Lemma 3.2]{CN19} in the $L^{\infty}$ setting).

\section{Backstepping transformation}\label{sect backstepping}

In this section, we use a Volterra transformation of the second kind to transform our initial system \eqref{syst} into a system with a simpler coupling structure, this is the so-called backstepping method for partial differential equations.
The content of this section is quite standard by now (yet, formulated differently here), see for instance \cite[Section 2.2]{HVDMK19} (or \cite[Section 2]{CN19}).

\subsection{Removal of the diagonal terms}\label{sect remov diag}

First of all, we perform a simple preliminary transformation in order to remove the diagonal terms in $M$.
This is only a technical step, which is nevertheless necessary in view of the existence of the transformation that we will use in the next section, see Remark \ref{rem CN solv kern} below.
For convenience, we introduce the set
$$\mathcal{M}=\ens{M \in L^{\infty}(0,1)^{n \times n} \st m_{ii}=0, \quad \forall i \in \ens{1,\ldots,n}}.$$

\begin{proposition}\label{prop remove diag}
There exists a map $\Psi:L^{\infty}(0,1)^{n \times n} \longrightarrow \mathcal{M}$ such that, for every $M \in L^{\infty}(0,1)^{n \times n}$, we have
$$
(\Lambda,M,Q,-)
\equi
(\Lambda,\Psi(M),Q,-).
$$
\end{proposition}

\begin{proof}
\begin{itemize}
\item
We are going to use the spatial transformation
$$
\tilde{y}(t,x)=E(x)y(t,x),
$$
where $E=\diag(e_1,\ldots,e_n) \in W^{1,\infty}(0,1)^{n \times n}$ is the diagonal matrix whose entries are
$$
e_i(x)=\exp\left(-\int_0^x \frac{m_{ii}(\xi)}{\lambda_i(\xi)} \, d\xi\right).
$$
Clearly, this transformation is invertible on $L^2(0,1)^n$.

\item
Assume now that $y$ is a solution to the system $(\Lambda,M,Q,-)$ for some $y^0$ and $u$ and let us show that $\tilde{y}$ is then a solution to the system $(\Lambda,\Psi(M),Q,-)$ for some $\tilde{y}^0$ and $\tilde{u}$, where $\Psi(M)$ will be determined below.
We do it formally but this can be rigorously justified.

\begin{itemize}
\item
The initial data is obviously $\tilde{y}^0(x)=E(x)y^0(x)$.

\item
The boundary condition at $x=0$ is clearly satisfied since $\tilde{y}(t,0)=y(t,0)$.

\item
Looking at the boundary condition at $x=1$, the control $\tilde{u}$ is
$$\tilde{u}(t)=\tilde{y}_-(t,1)=E_{--}(1)y_-(t,1).$$

\item
Using the equation satisfied by $y$ and the fact that $\Lambda$ and $E$ commute, a computation shows that
$$
\pt{\tilde{y}}(t,x)+\Lambda(x) \px{\tilde{y}}(t,x)=
\left(E(x)M(x)+\Lambda(x) \px{E}(x)\right)y(t,x).
$$
Consequently, $\tilde{y}$ satisfies the desired equation if we take
$$
(\Psi(M))(x)=\left(E(x)M(x)+\Lambda(x) \px{E}(x)\right)E(x)^{-1}.
$$
\end{itemize}

Now that $\Psi$ is clearly identified, similar computations show that, conversely, if $\tilde{y}$ is a solution to the system $(\Lambda,\Psi(M),Q,-)$ for some $\tilde{y}^0$ and $\tilde{u}$, then $y$ is a solution to the system $(\Lambda,M,Q,-)$ for some $y^0$ and $u$.

\item
Finally, it is clear that $\Psi(M) \in \mathcal{M}$ by construction.

\end{itemize}

\end{proof}

\subsection{Backstepping transformation}

We now recall an important result from \cite{HVDMK19} and \cite{HDMVK16} that we present here using the notion of equivalent system.
To this end, we introduce the set
$$\mathcal{F}=\ens{A \in L^{\infty}(0,1)^{n \times n} \st A_{-+}=A_{+-}=0}.$$

\begin{theorem}\label{thm backstepping}
For every $A \in \mathcal{F}$, there exists a map $\Gamma_A:\mathcal{M} \longrightarrow L^{\infty}(0,1)^{n \times m}$ such that, for every $M \in \mathcal{M}$, we have
$$
(\Lambda,M,Q,-)
\equi
(\Lambda,-,Q,\Gamma_A(M)).
$$
\end{theorem}

\begin{proof}
\begin{itemize}
\item
We are going to use the spatial transformation
$$
\tilde{y}(t,x)=y(t,x)-\int_0^x K(x,\xi)y(t,\xi) \, d\xi,
$$
where $K \in L^{\infty}(\Tau)^{n \times n}$ and $\Tau$ is the triangle
$$\Tau=\ens{(x,\xi) \in (0,1)\times(0,1) \st x>\xi}.$$

This transformation is always invertible on $L^2(0,1)^n$ since it is a Volterra transformation of the second kind (see e.g. \cite[Theorem 2.5]{Hoc73}).

\item
Assume now that $y$ is a solution to the system $(\Lambda,M,Q,-)$ for some $y^0$ and $u$ and let us show that $\tilde{y}$ is then a solution to the system $(\Lambda,-,Q,\Gamma_A(M))$ for some $\tilde{y}^0$ and $\tilde{u}$, where $\Gamma_A(M)$ will be determined below.
We do it formally but this can be rigorously justified.

\begin{itemize}
\item
The initial data is obviously $\tilde{y}^0(x)=y^0(x)-\int_0^x K(x,\xi) y^0(\xi) \, d\xi$.

\item
The boundary condition at $x=0$ is clearly satisfied since $\tilde{y}(t,0)=y(t,0)$.

\item
Looking at the boundary condition at $x=1$, the control $\tilde{u}$ is
\begin{equation}\label{def utilde}
\tilde{u}(t)
=\tilde{y}_-(t,1)
=y_-(t,1)-\int_0^1 H(\xi)y(t,\xi) \, d\xi,
\end{equation}
where $H(\xi)=
\begin{pmatrix}
K_{--}(1,\xi) & K_{-+}(1,\xi)
\end{pmatrix}
$.

\item
Using the equation satisfied by $y$, integrating by parts, and using the boundary condition satisfied by $y$ at $x=0$, we have
\begin{align*}
\MoveEqLeft[2]
\pt{\tilde{y}}(t,x)+\Lambda(x) \px{\tilde{y}}(t,x)=
\\
&-\int_0^x \left(
\Lambda(x)\px{K}(x,\xi)
+\pxi{K}(x,\xi)\Lambda(\xi)
+K(x,\xi)\left(\pxi{\Lambda}(\xi)+M(\xi)\right)
\right)y(t,\xi) \, d\xi
\\
&+\left(M(t,x)+K(x,x)\Lambda(x)-\Lambda(x)K(x,x)\right)y(t,x)
\\
&-K(x,0)\Lambda(0)\begin{pmatrix} \Id_{\R^{m \times m}} \\ Q \end{pmatrix} y_-(t,0).
\end{align*}
Consequently, $\tilde{y}$ satisfies the desired equation if we take
\begin{equation}\label{def GammaA}
(\Gamma_A(M))(x)
=
-K(x,0)\Lambda(0)\begin{pmatrix} \Id_{\R^{m \times m}} \\ Q \end{pmatrix},
\end{equation}
and provided that the kernel $K$ satisfies the so-called kernel equations:
\begin{equation}\label{kern equ}
\begin{dcases}
\Lambda(x)\px{K}(x,\xi)
+\pxi{K}(x,\xi)\Lambda(\xi)
+K(x,\xi)\left(\pxi{\Lambda}(\xi)+M(\xi)\right)=0,
\\
\Lambda(x)K(x,x)-K(x,x)\Lambda(x)=M(x).
\end{dcases}
\end{equation}
The existence of a solution to these equations will be discussed next.
\end{itemize}

Now that $\Gamma_A$ is clearly identified, similar computations show that, conversely, if $\tilde{y}$ is a solution to the system $(\Lambda,-,Q,\Gamma_A(M))$ for some $\tilde{y}^0$ and $\tilde{u}$, then $y$ is a solution to the system $(\Lambda,M,Q,-)$ for some $y^0$ and $u$.

\end{itemize}

\end{proof}

\begin{remark}\label{rem CN solv kern}
If we write the second condition of \eqref{kern equ} component-wise:
\begin{equation}\label{equ remov diag}
\left(\lambda_i(x)-\lambda_j(x)\right)k_{ij}(x,x)=m_{ij}(x),
\end{equation}
then we see that for $i=j$ we shall necessarily have $m_{ii}=0$.
Therefore, it is necessary that $M \in \mathcal{M}$ (otherwise the equation \eqref{equ remov diag}, and thus the kernel equations \eqref{kern equ}, have no solution).
This explains why we had to perform a preliminary transformation in Section \ref{sect remov diag} to reduce the general case to this one.
\end{remark}

From \cite[Section VI]{HDMVK16}, we know that the kernel equations \eqref{kern equ} have a solution (see also \cite[Remark A.2]{HVDMK19} to see how to deal with space-varying speeds).
More precisely, we can extract the following result:
\begin{theorem}\label{thm kern}
For every $A \in \mathcal{F}$, for every $M \in \mathcal{M}$, there exists a unique solution $K \in L^{\infty}(\Tau)^{n \times n}$ to the kernel equations \eqref{kern equ} with:
\begin{itemize}
\item
For every $i,j \in \ens{1,\ldots,m}$:
\begin{equation}\label{additional conditions 1}
\begin{aligned}
& k_{ij}(1,\xi)=a_{ij}(\xi), & & \text{ if } i>j, \\
& k_{ij}(x,0)=a_{ij}(x), & & \text{ if } i \leq j.
\end{aligned}
\end{equation}

\item
For every $i,j \in \ens{m+1,\ldots,n}$:
\begin{equation}\label{additional conditions 2}
\begin{aligned}
& k_{ij}(x,0)=a_{ij}(x), & & \text{ if } i \geq j, \\
& k_{ij}(1,\xi)=a_{ij}(\xi), & & \text{ if } i<j.
\end{aligned}
\end{equation}
\end{itemize}
Moreover, we have the following additional regularities:
\begin{gather*}
K \in C^0((0,1];L^2(0,x)^{n \times n}),
\quad K(x,\cdot) \in L^{\infty}(0,x)^{n \times n}, \quad \forall x \in (0,1].
\\
K \in C^0([0,1);L^2(\xi,1)^{n \times n}),
\quad K(\cdot,\xi) \in L^{\infty}(\xi,1)^{n \times n}, \quad \forall \xi \in [0,1).
\end{gather*}

\end{theorem}

As before, the notion of solution is to be understood in the sense of solution along the characteristics.
By $K \in C^0((0,1];L^2(0,x)^{n \times n})$ we mean that $\norm{K(x_n,\cdot)-K(x,\cdot)}_{L^2(0,\min\ens{x_n,x})^{n \times n}} \to 0$ as $x_n \to x$, for every $x \in (0,1]$, with a similar definition for $K \in C^0([0,1);L^2(\xi,1)^{n \times n})$.
Despite not mentioned in the literature, these important regularities can be deduced from the system of integral equations satisfied by the kernel.
In particular, it shows that $H$ and $\Gamma_A(M)$ defined in \eqref{def utilde} and \eqref{def GammaA} have the following regularities:
$$H \in L^{\infty}(0,1)^{m \times n}, \quad \Gamma_A(M) \in L^{\infty}(0,1)^{n \times m}.$$



\begin{remark}
The set $\mathcal{F}$ corresponds to the set of boundary conditions that are free to choose for the kernel equations.
The freedom for the boundary condition \eqref{additional conditions 1} was already used in the works \cite{HDM15,HDMVK16,HVDMK19} in order to give to $(\Gamma_A(M))_{--}$ a structure of strictly lower triangular matrix.
However, in the present paper this will not be used and it is the other boundary condition \eqref{additional conditions 2} that will turn out to be essential (see Section \ref{sect reduc to casc} below).
\end{remark}

\section{Reduction of the boundary coupling matrix}\label{sect boundary coupling matrix}

In this section we perform some transformations to show that we can always assume that the boundary coupling matrix $Q$ is in canonical form.
More precisely, we prove the following result:

\begin{proposition}\label{prop reduc Q}
For every invertible upper triangular matrix $U \in \R^{m \times m}$ and every invertible lower triangular matrix $L \in \R^{p \times p}$, there exists a map $\Theta:L^{\infty}(0,1)^{n \times m} \longrightarrow L^{\infty}(0,1)^{n \times m}$ such that, for every $G \in L^{\infty}(0,1)^{n \times m}$, we have
$$
\left(\Lambda,-,Q,G\right)
\equi
\left(\Lambda,-,LQU,\Theta(G)\right).
$$
\end{proposition}

\begin{proof}
\begin{itemize}
\item
For any $i,j \in \ens{1,\ldots,n}$, we denote by $\zeta_{ij}$ the solution to the ODE
$$
\begin{dcases}
\dds \zeta_{ij}(s)= \frac{\lambda_j(\zeta_{ij}(s))}{\lambda_i(s)}, \quad \forall s \in \R, \\
\zeta_{ij}(0)=0.
\end{dcases}
$$

\item
We first prove that, for every invertible upper triangular matrix $U \in \R^{m \times m}$,  there exists a map $\Theta_{--}:L^{\infty}(0,1)^{m \times m} \longrightarrow L^{\infty}(0,1)^{m \times m}$ such that, for every $G \in L^{\infty}(0,1)^{n \times m}$, we have
$$
(\Lambda,-,Q,G)
\equi
\left(\Lambda,-,QU,
\begin{pmatrix}
\Theta_{--}(G_{--}) \\
G_{+-}U
\end{pmatrix}
\right).
$$
To this end, we are going to use the spatial transformation
\begin{equation}\label{transfo Q to QU}
\tilde{y}_i(t,x)=
\begin{dcases}
\sum_{k=i}^m u^{ik} y_k(t,\zeta_{ik}(x)) & \text{ for } i \in \ens{1,\ldots,m}, \\
y_i(t,x) & \text{ for } i \in \ens{m+1,\ldots,n},
\end{dcases}
\end{equation}
where $U^{-1}=(u^{ik})_{1 \leq i,k \leq m}$.
Let us first show that this transformation is well defined and invertible.
We can check that, for $i \leq k \leq m$, we have (recall \eqref{def phi})
\begin{equation}\label{zetaik}
\zeta_{ik}(x)=\phi_k^{-1}(\phi_i(x)).
\end{equation}
In particular, for such indices, $\zeta_{ik}$ is a $C^1$-diffeomorphism from $(0,1)$ to a subset of $(0,1)$ and thus the transformation \eqref{transfo Q to QU} is well defined on $L^2(0,1)^n$.
Besides, using the property $\zeta_{kj}(\zeta_{ik}(x))=\zeta_{ij}(x)$ for $i \leq k \leq j$, we can check that its inverse is given by
$$
y_k(t,x)=
\begin{dcases}
\sum_{j=k}^m u_{kj} \tilde{y}_j(t,\zeta_{kj}(x)) & \text{ for } k \in \ens{1,\ldots,m}, \\
\tilde{y}_k(t,x) & \text{ for } k \in \ens{m+1,\ldots,n}.
\end{dcases}
$$

\item
Assume now that $y$ is a solution to the system $(\Lambda,-,Q,G)$ for some $y^0$ and $u$ and let us show that $\tilde{y}$ is then a solution to the system $(\Lambda,-,QU,\widetilde{G})$ for some $\tilde{y}^0$ and $\tilde{u}$, where
$$
\widetilde{G}=
\begin{pmatrix}
\Theta_{--}(G_{--}) \\
G_{+-}U
\end{pmatrix},
$$
and where $\Theta_{--}(G_{--})$ will be determined below.
Once again, we do it formally but this can be rigorously justified.

\begin{itemize}
\item
The initial data is obviously
$$
\tilde{y}_i^0(x)=
\begin{dcases}
\sum_{k=i}^m u^{ik} y_k^0(\zeta_{ik}(x)) & \text{ for } i \in \ens{1,\ldots,m}, \\
y_i^0(x) & \text{ for } i \in \ens{m+1,\ldots,n}.
\end{dcases}
$$

\item
The boundary condition at $x=0$ is clearly satisfied since $\tilde{y}_+=y_+$ and $\tilde{y}_-(t,0)=U^{-1}y_-(t,0)$.

\item
Looking at the boundary condition at $x=1$, the control $\tilde{u}$ is
$$
\tilde{u}_i(t)
=\tilde{y}_i(t,1)
=\sum_{k=i}^m u^{ik} y_k(t,\zeta_{ik}(1)), \quad \forall i \in \ens{1,\ldots,m}.
$$

\item
It is clear that $\tilde{y}_+=y_+$ satisfies the desired equation.
Let us now fix $i \in \ens{1,\ldots,m}$.
A computation shows that
\begin{align*}
\MoveEqLeft[9]
\pt{\tilde{y}_i}(t,x)
+\lambda_i(x)\px{\tilde{y}_i}(t,x)
-\sum_{j=1}^m \tilde{g}_{ij}(x)\tilde{y}_j(t,0)
=
\\
&\sum_{k=i}^m u^{ik} \left(
-\lambda_k(\zeta_{ik}(x))
+\lambda_i(x)\px{\zeta_{ik}}(x)
\right)\px{y}(t,\zeta_{ik}(x))
\\
+&\sum_{\ell=1}^m \left(
\sum_{k=i}^m u^{ik} g_{k\ell}(\zeta_{ik}(x))
-\sum_{j=1}^{\ell} \tilde{g}_{ij}(x)u^{j\ell}
\right) y_{\ell}(t,0).
\end{align*}

Consequently, $\tilde{y}_i$ satisfies the desired equation, provided that
$$
\sum_{k=i}^m u^{ik} g_{k\ell}(\zeta_{ik}(x))
-\sum_{j=1}^{\ell} \tilde{g}_{ij}(x)u^{j\ell}=0, \quad \forall \ell \in \ens{1,\ldots,m}.
$$
This uniquely determines $\tilde{g}_{ij}$ for $i,j \in \ens{1,\ldots,m}$ (and thus $\Theta_{--}$):
$$
\tilde{g}_{ij}(x)=\sum_{\ell=1}^j \left(\sum_{k=i}^m u^{ik} g_{k\ell}(\zeta_{ik}(x))\right) u_{\ell j}.
$$


\end{itemize}

Now that $\Theta_{--}$ is clearly identified, similar computations show that, conversely, if $\tilde{y}$ is a solution to the system $(\Lambda,-,QU,\widetilde{G})$ for some $\tilde{y}^0$ and $\tilde{u}$, then $y$ is a solution to the system $(\Lambda,-,Q,G)$ for some $y^0$ and $u$.

\item
Similarly, we can prove that, for every invertible lower triangular matrix $L \in \R^{p \times p}$,  there exists a map $\Theta_{+-}:L^{\infty}(0,1)^{p \times m} \longrightarrow L^{\infty}(0,1)^{p \times m}$ such that, for every $G \in L^{\infty}(0,1)^{n \times m}$, we have
$$
(\Lambda,-,Q,G)
\equi
\left(\Lambda,-,LQ,
\begin{pmatrix}
G_{--} \\
\Theta_{+-}(G_{+-})
\end{pmatrix}
\right).
$$
This can be done using the spatial transformation
$$
\tilde{y}_i(t,x)=
\begin{dcases}
y_i(t,x) & \text{ for } i \in \ens{1,\ldots,m}, \\
\sum_{k=m+1}^i \ell_{i-m,k-m} y_k\left(t,\zeta_{ik}(x)\right) & \text{ for } i \in \ens{m+1,\ldots,n},
\end{dcases}
$$
(\eqref{zetaik} is still valid for the indices considered) where $L=(\ell_{ij})_{1 \leq i,j \leq p}$ and taking
$$\tilde{g}_{ij}(x)=\sum_{k=m+1}^i \ell_{i-m,k-m} g_{kj}(\zeta_{ik}(x)),$$
for $i \in \ens{m+1,\ldots,n}$ and $j \in \ens{1,\ldots,m}$, where $\widetilde{G}$ denotes the matrix
$$
\widetilde{G}=
\begin{pmatrix}
G_{--} \\
\Theta_{+-}(G_{+-})
\end{pmatrix}.
$$

\end{itemize}

\end{proof}

\section{Smallest value of the minimal null control time}\label{sect infT}

Thanks to the result of previous section it is from now on sufficient to consider boundary coupling matrices which are in canonical form.
This is a big step forward, which already allows us to characterize the smallest value of the minimal null control time.

\subsection{Characterization for systems $(\Lambda,-,Q,G)$}

We start with systems of the form $(\Lambda,-,Q,G)$, we will discuss in the next section how to deduce the corresponding result for the initial system $(\Lambda,M,Q,-)$.

\begin{theorem}\label{thm NC}
Let $Q^0 \in \R^{p \times m}$ be in canonical form, $G \in L^{\infty}(0,1)^{n \times m}$ and $T>0$ be fixed.
\begin{enumerate}[(i)]
\item\label{thm NC i1}
If the system $(\Lambda,-,Q^0,G)$ is null controllable in time $T$, then necessarily
\begin{equation}\label{cond T infT}
T \geq \max\ens{
\max_{k \in \ens{1,\ldots,\rho}}
T_{m+r_k}(\Lambda)+T_{c_k}(\Lambda), \quad
T_{m+1}(\Lambda)
, \quad
T_m(\Lambda)
}.
\end{equation}

\item\label{thm NC i2}
If $T$ satisfies the condition \eqref{cond T infT}, then the system $(\Lambda,-,Q^0,-)$ (i.e. with $G=0$) is null controllable in time $T$ with control $u=0$.
\end{enumerate}
\end{theorem}

As for Theorem \ref{main thm}, we use the convention that the undefined quantities are simply not taken into account, which means that the condition \eqref{cond T infT} is reduced to $T \geq \max\ens{T_{m+1}(\Lambda), \quad T_m(\Lambda)}$ when $\rho=0$ (i.e. when $Q^0=0$).

This result shows in particular that the smallest value that $\Tinf(\Lambda,-,Q^0,G)$ can take with respect to $G \in L^{\infty}(0,1)^{n \times m}$ is equal to the quantity on the right-hand side of the inequality in \eqref{cond T infT}.
This can be extended to arbitrary boundary coupling matrices thanks to Proposition \ref{prop reduc Q}.

\begin{proof}[Proof of Theorem \ref{thm NC}]
We use the ideas of the proof of \cite[Lemma 3.3]{HO20-pre}.
\begin{enumerate}[1)]
\item
We first show that it is necessary that
$$T \geq \max\ens{T_{m+1}(\Lambda), \quad T_m(\Lambda)}.$$
We point out that for this first step there is no need to assume that $Q^0$ is in canonical form.
Assume that $T<\max\ens{T_{m+1}(\Lambda), \quad T_m(\Lambda)}$.
Then, there exists $i \in \ens{1,\ldots,n}$ such that $T<T_i(\Lambda)$.
Let $\omega_i$ be the open subset defined by
\begin{equation}\label{def omegai}
\omega_i=\ens{x \in (0,1) \quad \middle| \quad \ssin_i(T,x)<0}.
\end{equation}
Then, we have (see \eqref{def Ti gen}, \eqref{inv sin} and \eqref{ssin monotonicity})
$$T<T_i(\Lambda) \quad \Longleftrightarrow \quad \omega_i \neq \emptyset.$$
For $x \in \omega_i$, the null controllability condition $y_i(T,x)=0$ is equivalent to (see \eqref{sol char def 1})
$$
0=
y^0_i\left(\chi_i(0;T,x)\right)
+\int_0^T \sum_{j=1}^m g_{ij}\left(\chi_i(s;T,x)\right) y_j(s,0) \, ds.
$$
Since $y^0_i \in L^2(0,1)$ is arbitrary and $x \in \omega_i \mapsto \chi_i(0;T,x)$ is a $C^1$-diffeomorphism (its inverse is given by $\xi \longmapsto \chi_i(T;0,\xi)$ thanks to \eqref{chi invariance}), this shows that the bounded linear operator $K:L^2(0,T)^m \longrightarrow L^2(\omega_i)$ defined by
$$
(Kh)(x)
=-\int_0^T \sum_{j=1}^m g_{ij}\left(\chi_i(s;T,x)\right) h_j(s) \, ds,
$$
is surjective.
This is impossible since its range is clearly a subset of $L^{\infty}(\omega_i)$, which is a proper subset of $L^2(\omega_i)$.

\item
Suppose now that $\rho \neq 0$ (otherwise we are done) and that $T$ is such that
$$
\max\ens{T_{m+1}(\Lambda), \quad T_m(\Lambda)} \leq T<T_{m+r_{k_0}}(\Lambda)+T_{c_{k_0}}(\Lambda),
$$
where $k_0 \in \ens{1,\ldots,\rho}$ is any index such that
$$T_{m+r_{k_0}}(\Lambda)+T_{c_{k_0}}(\Lambda)=\max_{k \in \ens{1,\ldots,\rho}}
T_{m+r_k}(\Lambda)+T_{c_k}(\Lambda).$$
We have seen in the previous step that the condition $T \geq \max\ens{T_{m+1}(\Lambda), \quad T_m(\Lambda)}$ means that all the subsets $\omega_i$ defined in \eqref{def omegai} are empty.
In particular (recall also \eqref{ssin monotonicity}),
$$\ssin_{m+r_{k_0}}(T,x)>0, \quad \forall x \in (0,1).$$
Therefore, the null controllability condition $y_{m+r_{k_0}}(T,x)=0$ is equivalent to (see \eqref{sol char def 1} and recall that $Q^0$ is in canonical form)
\begin{equation}\label{NC cond first step}
0=y_{c_{k_0}}(\ssin_{m+r_{k_0}}(T,x), 0)
+\int_{\ssin_{m+r_{k_0}}(T,x)}^T \sum_{j=1}^m g_{m+r_{k_0},j}\left(\chi_{m+r_{k_0}}(s;T,x)\right) y_j(s,0) \, ds.
\end{equation}
Let us now introduce the open subset $\tilde{\omega}$ defined by
$$
\tilde{\omega}
=\ens{x \in (0,1) \quad \middle| \quad \ssin_{c_{k_0}}(\ssin_{m+r_{k_0}}(T,x), 0)<0}.
$$
Using that $T_{m+r_{k_0}}(\Lambda)+T_{c_{k_0}}(\Lambda)=\ssout_{m+r_{k_0}}(\ssout_{c_{k_0}}(0,1), 0)$ (see \eqref{def Ti gen} and \eqref{explicit formula ssin}), we can show by a similar reasoning as in the first step that
$$
T<T_{m+r_{k_0}}(\Lambda)+T_{c_{k_0}}(\Lambda)
\quad \Longleftrightarrow \quad
\tilde{\omega} \neq \emptyset.
$$
For $x \in \tilde{\omega}$, the identity \eqref{NC cond first step} becomes (see \eqref{sol char def 1})
\begin{align*}
\MoveEqLeft[8]
0=y^0_{c_{k_0}}(\chi_{c_{k_0}}(0;\ssin_{m+r_{k_0}}(T,x), 0))
\\
&+\sum_{j=1}^m \int_{\ssinb_{c_{k_0}}(\ssin_{m+r_{k_0}}(T,x),0)}^{\ssin_{m+r_{k_0}}(T,x)} g_{c_{k_0} j}\left(\chi_{c_{k_0}}(s; \ssin_{m+r_{k_0}}(T,x), 0)\right)y_j\left(s,0\right) \, ds
\\
&+\int_{\ssin_{m+r_{k_0}}(T,x)}^T \sum_{j=1}^m g_{m+r_{k_0},j}\left(\chi_{m+r_{k_0}}(s;T,x)\right) y_j(s,0) \, ds.
\end{align*}
This leads to a contradiction by using the same argument as at the end of the first step.

\item
Finally, it is not difficult to see from \eqref{sol char def 1} that, when $G=0$, the control $u=0$ brings the solution of the system $(\Lambda,-,Q^0)$ to zero in any time $T$ satisfying \eqref{cond T infT}.

\end{enumerate}

\end{proof}

\subsection{Proof of the first part of Theorem \ref{main thm}}\label{sect proof first part}

Let us now show how the previous results yield the desired characterization of the smallest minimal null control time for the initial system $(\Lambda,M,Q,-)$.

\begin{proof}[Proof of item \ref{item infT} of Theorem \ref{main thm}]
\begin{itemize}
\item
Let $M \in L^{\infty}(0,1)^{n \times n}$ and $Q \in \R^{p \times m}$ be fixed.
Let $T>0$ be such that the system $(\Lambda,M,Q,-)$ is null controllable in time $T$.

\begin{itemize}
\item
By Proposition \ref{prop remove diag} and Theorem \ref{thm backstepping}, there exists $G \in L^{\infty}(0,1)^{n \times m}$ such that the system $(\Lambda,-,Q,G)$ is null controllable in time $T$.

\item
From Proposition \ref{prop reduc Q}, there exists $\widetilde{G} \in L^{\infty}(0,1)^{n \times m}$ such that  the system $(\Lambda,-,Q^0,\widetilde{G})$ is null controllable in time $T$, where $Q^0$ is the canonical form of $Q$.

\item
By item \ref{thm NC i1} of Theorem \ref{thm NC} we obtain that $T$ has to satisfy the condition \eqref{cond T infT}.
\end{itemize}

This establishes the following lower bound:
$$
\Tinf(\Lambda,M,Q) \geq
\max\ens{
\max_{k \in \ens{1,\ldots,\rho}}
T_{m+r_k}(\Lambda)+T_{c_k}(\Lambda), \quad
T_{m+1}(\Lambda)
, \quad
T_m(\Lambda)
},
$$
valid for every $M \in L^{\infty}(0,1)^{n \times n}$.

\item
This lower bound is reached for $M=0$, this follows from Theorem \ref{thm NC} and Proposition \ref{prop reduc Q} (using that $\Theta(0)=0$).
\end{itemize}

\end{proof}

\subsection{Comments on the case $M=0$}

Let us conclude this section with some interesting remarks on the case $M=0$.
For $M=0$, we can combine Theorem \ref{thm NC} with Proposition \ref{prop reduc Q} (with $G=0$, in which case their proofs are greatly simplified) to obtain a completely different proof of \cite[Theorems 1 and 2]{Wec82}.
Our proof has several advantages.
Firstly, we directly obtain a more explicit expression of the minimal null control time (see e.g. \cite[Remark 1.15]{HO21}).
On the other hand, we do not need to use the so-called duality and we are able to obtain an explicit control.
More precisely, we can extract the following result from item \ref{thm NC i2} of Theorem \ref{thm NC} and the proof of Proposition \ref{prop reduc Q}:


\begin{proposition}\label{prop explicit cont when M=0}

Let $Q \in \R^{p \times m}$ and $T$ satisfy \eqref{cond T infT}.
Then, the system $(\Lambda,-,Q,-)$ is finite-time stabilizable with settling time $T$, with the following explicit feedback law:
\begin{equation}\label{explicit feedback}
u_i(t)
=-\sum_{k=i+1}^m u^{ik} y_k(t,\zeta_{ik}(1)), \quad i \in \ens{1,\ldots,m},
\end{equation}
where $U^{-1}=(u^{ik})_{1 \leq i,k \leq m}$ and $U$ is any matrix $U$ coming from the $LCU$ decomposition of $Q$.

\end{proposition}

We recall that the previous statement simply means that, if we replace the $i$-th component of $u$ by the right-hand side of the formula \eqref{explicit feedback} in the system \eqref{syst} (with $M=0$), then the corresponding solution satisfies $y(T,\cdot)=0$ for every $y^0 \in L^2(0,1)^n$.
We also recall that systems with such boundary conditions are well posed (see e.g. \cite[Section 3]{CN19} in the $L^{\infty}$ setting).

A similar result was obtained in the proof of \cite[Proposition 1.6]{CN19} when $Q \in \setCN$ (defined in \eqref{def setCN}), our result generalizes it to arbitrary $Q \in \R^{p \times m}$.
Let us illustrate with an example that the feedback law that we have obtained \eqref{explicit feedback} is also the same as in this reference when $Q \in \setCN$.

\begin{example}
Let us consider the $6 \times 6$ system used as example in \cite[p. 1155]{CN19}: we take $p=m=3$, the negative speeds are
$$\lambda_1=-4< \lambda_2=-2<\lambda_3=-1<0,$$
the positive speeds are arbitrary (subject to \eqref{hyp speeds}), and we take the boundary coupling matrix
$$
Q=\begin{pmatrix}
1 & -1 & -1 \\
1 & 0 & 2 \\
a & b & c
\end{pmatrix},
$$
where $a,b,c \in \R$ are arbitrary numbers.

Then, the Gaussian elimination easily shows that $Q \in \setCN$ and
$$
U=\begin{pmatrix}
1 & 1 & -2 \\
0 & 1 & -3 \\
0 & 0 & 1
\end{pmatrix},
\quad
U^{-1}=\begin{pmatrix}
1 & -1 & -1 \\
0 & 1 & 3 \\
0 & 0 & 1
\end{pmatrix}
.
$$
The feedback law is thus given by
$$
\begin{dcases}
y_1(t,1)=y_2\left(t,\frac{1}{2}\right)+y_3\left(t,\frac{1}{4}\right), \\
y_2(t,1)=-3y_3\left(t,\frac{1}{2}\right), \\
y_3(t,1)=0.
\end{dcases}
$$

\end{example}

\begin{remark}
Let us also add that another advantage of not using the duality is that it can be useful to deal with other functional settings (e.g. $C^1$, provided that the inequality in \eqref{cond T infT} is strict).
\end{remark}

\section{Reduction to a canonical system}\label{sect reduc to casc}

We are now left with the proof of the second part of Theorem \ref{main thm}, which is more difficult and require more work.

In this section, we will show how to use the canonical structure of the boundary coupling matrix to remove some coupling terms in the matrix $G_{+-}$.
For any $Q \in \R^{p \times m}$, we introduce the set
$$\mathcal{C}(Q)=\ens{G_{+-} \in L^{\infty}(0,1)^{p \times m} \st g_{m+i, c_k}=0, \quad \forall k \in \ens{1,\ldots,\rho}, \, \forall i \geq r_k}$$
($\mathcal{C}(0)=L^{\infty}(0,1)^{p \times m}$).

The goal of this section is to prove the following result:

\begin{proposition}\label{prop Gtilde}
Assume that $Q^0 \in \R^{p \times m}$ is in canonical form.
Then, there exists a map $\Upsilon:L^{\infty}(0,1)^{p \times m} \longrightarrow \mathcal{C}(Q^0)$ such that, for every $G \in L^{\infty}(0,1)^{n \times m}$, we have
\begin{equation}\label{equiv syst reduc Gpm}
(\Lambda,-,Q^0,G)
\equi
\left(\Lambda,-,Q^0,\begin{pmatrix} G_{--} \\ \Upsilon(G_{+-}) \end{pmatrix} \right).
\end{equation}

\end{proposition}

\begin{proof}
We assume that $\rho \neq 0$ since otherwise there is nothing to prove.
Reproducing the proof of Theorem \ref{thm backstepping} with the kernel
$$
K=\begin{pmatrix}
0 & 0 \\
0 & K_{++}
\end{pmatrix},
$$
we see that we have \eqref{equiv syst reduc Gpm} if we take
$$
(\Upsilon(G_{+-}))(x)=
G_{+-}(x)-K_{++}(x,0)\Lambda_{++}(0)Q^0
-\int_0^x K_{++}(x,\xi)G_{+-}(\xi) \, d\xi,
$$
and provided that $K_{++}$ satisfies
$$
\begin{dcases}
\Lambda_{++}(x)\px{K_{++}}(x,\xi)
+\pxi{K_{++}}(x,\xi)\Lambda_{++}(\xi)
+K_{++}(x,\xi)\pxi{\Lambda_{++}}(\xi)=0,
\\
\Lambda_{++}(x)K_{++}(x,x)-K_{++}(x,x)\Lambda_{++}(x)=0.
\end{dcases}
$$
This is an uncoupled system with many solutions (as we already know from Theorem \ref{thm kern}).
Let us find a particular one that guarantees that $\Upsilon(G_{+-}) \in \mathcal{C}(Q^0)$.
Let $i,j \in \ens{m+1,\ldots,n}$ be fixed.
The equation for $k_{ij}$ is simply
\begin{equation}\label{kern equ pp}
\begin{dcases}
\lambda_i(x) \px{k_{ij}}(x,\xi)
+\pxi{k_{ij}}(x,\xi)\lambda_j(\xi)
+k_{ij}(x,\xi)\pxi{\lambda_j}(\xi)
=0,
\\
k_{ij}(x,x)=0, \quad \text{ if } i \neq j.
\end{dcases}
\end{equation}
Let $s \mapsto \zeta_{ij}(s;x,\xi)$ be the associated characteristic passing through $(x,\xi)$:
$$
\begin{dcases}
\pas{\zeta_{ij}}(s;x,\xi)= \frac{\lambda_j(\zeta_{ij}(s;x,\xi))}{\lambda_i(s)}, \\
\zeta_{ij}(x;x,\xi)=\xi.
\end{dcases}
$$
The solutions to \eqref{kern equ pp} are explicit:
\begin{itemize}
\item
If $i \geq j$, then there exists a unique solution to \eqref{kern equ pp} which satisfies $k_{ij}(x,0)=a_{ij}(x)$ ($a_{ij} \in L^{\infty}(0,1)$ is arbitrary) and it is given by
$$
k_{ij}(x,\xi)=
\begin{dcases}
a_{ij}(\ssin_{ij}(x,\xi)) \frac{\lambda_j(0)}{\lambda_j(\xi)} & \text{ if } \xi<\zeta_{ij}(x;0,0), \\
0 & \text{ if } \xi>\zeta_{ij}(x;0,0),
\end{dcases}
$$
where $\ssin_{ij}(x,\xi) \in (0,x)$ is the unique solution to
$$\zeta_{ij}(\ssin_{ij}(x,\xi);x,\xi)=0.$$

\item
If $i<j$, then there exists a unique solution to \eqref{kern equ pp} which satisfies $k_{ij}(1,\xi)=a_{ij}(\xi)$ ($a_{ij} \in L^{\infty}(0,1)$ is arbitrary) and it is given by
$$
k_{ij}(x,\xi)=
\begin{dcases}
a_{ij}(\zeta_{ij}(1;x,\xi)) \frac{\lambda_j(\zeta_{ij}(1;x,\xi))}{\lambda_j(\xi)} & \text{ if } \xi<\zeta_{ij}(x;1,1), \\
0 & \text{ if } \xi>\zeta_{ij}(x;1,1).
\end{dcases}
$$
\end{itemize}

We choose $a_{ij}=0$ for $i<j$, so that $k_{ij}=0$ for such indices.
Let us now fix the remaining $a_{ij}$ to ensure that $\Upsilon(G_{+-}) \in \mathcal{C}(Q^0)$.
To this end, we fix $i \in \ens{m+1,\ldots,n}$ such that $E_i \neq \emptyset$, where
$$E_i=\ens{\alpha \in \ens{1,\ldots,\rho} \st m+r_{\alpha} \leq i}.$$
The $(i,c_{\alpha})$-th entry of $\Upsilon(G_{+-})$ is equal to zero if, and only if,
$$
0=
g_{i c_{\alpha}}(x)
-\sum_{\ell=m+1}^n k_{i \ell}(x,0) \lambda_{\ell}(0)q^0_{\ell-m, c_{\alpha}}
-\int_0^x \sum_{\ell=m+1}^n k_{i \ell} (x,\xi) g_{\ell, c_{\alpha}}(\xi) \, d\xi.
$$
Using the explicit formulas for $k_{ij}$ and the assumption that $Q^0$ is in canonical form, for $\alpha \in E_i$ this identity is equivalent to
$$
0=
g_{i c_{\alpha}}(x)
-a_{i, m+r_{\alpha}}(x) \lambda_{m+r_{\alpha}}(0)
-\sum_{\ell=m+1}^i \int_0^{\zeta_{i \ell}(x;0,0)}
a_{i\ell}(\ssin_{i \ell}(x,\xi)) \frac{\lambda_{\ell}(0)}{\lambda_{\ell}(\xi)}
 g_{\ell, c_{\alpha}}(\xi) \, d\xi.
$$
Using the change of variable $\theta \mapsto \xi=\zeta_{i \ell}(x;\theta,0)$ with the property
$$\ssin_{i \ell}(x,\zeta_{i \ell}(x;\theta,0))=\ssin_{i \ell}(\theta,0)=\theta,$$
and isolating the terms for $\ell=m+r_\beta$ with $\beta \in E_i$, this gives the following system of Volterra equations of the second kind:
\begin{equation}\label{volt equ for a}
a_{i, m+r_{\alpha}}(x)\lambda_{m+r_{\alpha}}(0)
+\sum_{\beta \in E_i} \int_0^x a_{i, m+\beta}(\theta) h_{i, \alpha, m+\beta}(x,\theta) \, d\theta
=f_{i \alpha}(x), \quad \alpha \in E_i,
\end{equation}
with $L^{\infty}$ kernel
$$
h_{i \alpha \ell}(x,\theta)=
\frac{\lambda_{\ell}(0)}{\lambda_{\ell}(\zeta_{i \ell}(x;\theta,0))}
 g_{\ell c_{\alpha}}(\zeta_{i \ell}(x;\theta,0)) \frac{\partial \zeta_{i \ell}}{\partial \theta}(x;\theta,0),
$$
and $L^{\infty}$ right-hand side
\begin{gather*}
f_{i \alpha}(x)=
g_{i c_{\alpha}}(x)
-\sum_{\ell \in F_i} \int_0^x
a_{i\ell}(\theta) h_{i \alpha \ell}(x,\theta) \, d\theta,
\\
F_i=\ens{\ell \in \ens{m+1,\ldots,i} \st \ell \not\in \ens{m+r_1,\ldots,m+r_{\rho}} \text{ or } \ell \in \ens{m+r_{\beta} \st \beta \not\in E_i}}.
\end{gather*}
Setting
$$a_{i\ell}=0, \quad \forall \ell \in F_i,$$
we have $f_{i\alpha}=g_{i c_{\alpha}}$ and, once $f_{i\alpha}$ is known, the remaining values $a_{i, m+r_{\alpha}}$, $\alpha \in E_i$, are uniquely determined by solving the system \eqref{volt equ for a}.

\end{proof}

\begin{remark}\label{rem canonical form}
Let us point out that it is also possible to transform the matrix $G_{--}$ into a strictly lower triangular matrix by using the kernel
$$
K=\begin{pmatrix}
K_{--} & 0 \\
0 & 0
\end{pmatrix},
$$
and by appropriately choosing some boundary conditions for $K_{--}$ (this is the same proof as above with $Q^0=\Id$).
In summary, whatever $Q \in \R^{p \times m}$ and $G \in L^{\infty}(0,1)^{n \times m}$ are, we have shown that we can always find some transformations so that we are reduced to the case where:
\begin{itemize}
\item
$Q$ is in canonical form.

\item
$G_{--}$ is strictly lower triangular.

\item
$G_{+-} \in \mathcal{C}(Q)$.
\end{itemize}

Let us add that, in general, it is not possible to remove more terms by using some transformations (e.g. the backstepping method).
In other words, there is in general no simpler equivalent system.
An example has been detailed in Appendix \ref{sect counterexample}.
In this sense, systems $(\Lambda,-,Q,G)$ with the above structure could be called ``in canonical form''.
\end{remark}

\section{Reduction by compactness-uniqueness}\label{sect compactness-uniqueness}

In this section, we show that, even though we can not in general fully remove $G_{--}$ by using some transformations (Remark \ref{rem canonical form}), nevertheless, the two systems share the same minimal null control time:

\begin{theorem}\label{thm remove gmm}
Let $Q \in \R^{p \times m}$ be fixed.
For every $G \in L^{\infty}(0,1)^{n \times m}$, we have
\begin{equation}\label{Tinf Gmm}
\Tinf(\Lambda,-,Q,G)=
\Tinf\left(\Lambda,-,Q,
\begin{pmatrix}
0 \\
G_{+-}
\end{pmatrix}
\right).
\end{equation}
\end{theorem}

\begin{remark}\label{rem non equivalent}
Let us emphasize once again that it is impossible to prove Theorem \ref{thm remove gmm} by using some transformations to pass from one system to the other (e.g. backstepping).
In other words, these two systems are in general not equivalent (in the sense of Definition \ref{def syst equiv}).
Therefore, a different method is necessary to prove Theorem \ref{thm remove gmm}.
We will do it thanks to a compactness-uniqueness method adapted to the null controllability property.
\end{remark}

\subsection{A compactness-uniqueness method for the null controllability}\label{sect abstract}

We will present here a general compactness-uniqueness method adapted to the null controllability property.
We will see in the next section how to use it in order to obtain Theorem \ref{thm remove gmm}.

First of all, let us briefly recall some basic facts about abstract linear control systems.
All along this section, $H$ and $U$ are two complex Hilbert spaces, $A:\dom{A} \subset H \longrightarrow H$ is the generator of a $C_0$-semigroup $(S(t))_{t \geq 0}$ on $H$ and $B \in \lin{U,\dom{A^*}'}$.
Here and in what follows, $E'$ denotes the anti-dual of the complex space $E$, that is the complex (Banach) space of all continuous conjugate linear forms.
We will use the convention that an inner product of a complex Hilbert space is conjugate linear in its second argument.
One of the reason why we have to consider complex (and not real) spaces is because we will use below a condition involving the spectral elements of the operator $A$, we will explain how to deal with real Banach spaces in practice at the end of this section in Remark \ref{rem complexification}.

Let us now consider the evolution problem associated with the pair $(A,B)$, i.e.
\begin{equation}\label{abst syst}
\begin{dcases}
\ddt y(t)= Ay(t)+B u(t), \quad t \in (0,T), \\
y(0)=y^0,
\end{dcases}
\end{equation}
where $T>0$, $y(t)$ is the state at time $t$, $y^0$ is the initial data and $u(t)$ is the control at time $t$.

Let us recall a standard procedure to define a notion of solution in $H$ to \eqref{abst syst} for non smooth functions.
We formally multiply \eqref{abst syst} by a smooth function $z$, integrate over an arbitrary time interval $(0,\tau) \subset (0,T)$, perform an integraton by parts and use the adjoints to obtain the identity
$$
\ps{y(\tau)}{z(\tau)}{H}
-\ps{y^0}{z(0)}{H}
+\int_0^\tau \ps{y(t)}{-\ddt z(t)-A^*z(t)}{H} \, dt
=\int_0^\tau \ps{u(t)}{B^*z(t)}{U} \, dt.
$$
Particularizing this identity for the solution $z$ to the so-called adjoint system
\begin{equation}\label{adj syst}
\begin{dcases}
-\ddt z(t) =A^*z(t), \quad t \in (0,\tau), \\
z(\tau)=z^1,
\end{dcases}
\end{equation}
i.e. $z(t) = S(\tau-t)^*z^1$, where $z^1$ is arbitrary, this leads to the following notion of solution in $H$:

\begin{definition}\label{def sol abst}
Let $T>0$, $y^0 \in H$ and $u \in L^2(0,T;U)$ be fixed.
We say that a function $y:[0,T] \longrightarrow H$ is a solution to \eqref{abst syst} if $y \in C^0([0,T];H)$ and
\begin{equation}\label{def sol}
\ps{y(\tau)}{z^1}{H}
-\ps{y^0}{z(0)}{H}
=\int_0^\tau \ps{u(t)}{B^*z(t)}{U} \, dt,
\end{equation}
for every $\tau \in (0,T]$ and $z^1 \in \dom{A^*}$, where $z \in C^0([0,\tau];\dom{A^*})$ is the solution to the adjoint system \eqref{adj syst}.
\end{definition}

For the system \eqref{abst syst} to be well posed in this sense, the space $H$ has to satisfy some properties.

\begin{definition}\label{def admi}
We say that $H$ is an admissible subspace for the system $(A,B)$ if the following regularity property holds:
\begin{equation}\label{estim admis}
\forall \tau>0, \, \exists C>0, \quad \int_0^\tau \norm{B^*z(t)}_U^2 \, dt \leq C \norm{z^1}_H^2,
\quad \forall z^1 \in \dom{A^*},
\end{equation}
where $z \in C^0([0,\tau];\dom{A^*})$ is the solution to the adjoint system \eqref{adj syst}.

\end{definition}

We recall that, thanks to basic semigroup properties, it is equivalent to prove \eqref{estim admis} for one single $\tau>0$.

If $H$ is an admissible subspace for $(A,B)$, then the map
$$z^1 \in \dom{A^*} \mapsto \int_0^\tau \ps{u(t)}{B^*z(t)}{U} \, dt,$$
can be extended to a continuous conjugate linear form on $H$.
Thus, we have a natural definition for the map $\tau \in [0,T] \longmapsto y(\tau) \in H$ through the formula \eqref{def sol}.
It can be proved that this map is also continuous and that it depends continuously on $y^0$ and $u$ on compact time intervals (see e.g. \cite[Theorem 2.37]{Cor07}).
This establishes the so-called well-posedness of the abstract control system \eqref{abst syst} in $H$.

Now that we have a notion of continuous solution for the system \eqref{abst syst} in the space $H$, we can speak of its controllability properties in $H$.

\begin{definition}\label{def NC}
We say that the system \eqref{abst syst} is null controllable in time $T$ if, for every $y^0 \in H$, there exists $u \in L^2(0,T;U)$ such that the corresponding solution $y \in C^0([0,T];H)$ to the system \eqref{abst syst} satisfies
$$y(T)=0.$$
\end{definition}

It is also well known that controllability has a dual concept named observability.
We have the following characterization (see e.g. \cite[Theorem 2.44]{Cor07}):

\begin{theorem}\label{thm duality}
Let $T>0$ be fixed.
The system $(A,B)$ is null controllable in time $T$ if, and only if, there exists $C>0$ such that, for every $z^1 \in \dom{A^*}$,
$$\norm{z(0)}_H^2 \leq C\int_0^T \norm{B^*z(t)}_U^2 \, dt,$$
where $z \in C^0([0,T];\dom{A^*})$ is the solution to the adjoint system \eqref{adj syst} (with $\tau=T$).
\end{theorem}

After these basic reminders, we can now clearly introduce the general compactness-uniqueness result on which the proof of Theorem \ref{thm remove gmm} will rely on.

\begin{theorem}\label{thm compactness}
Let $H$ and $U$ be two complex Hilbert spaces.
Let $A:\dom{A} \subset H \longrightarrow H$ be the generator of a $C_0$-semigroup $(S(t))_{t \geq 0}$ on $H$ and let $B \in \lin{U,\dom{A^*}'}$.
We assume that $H$ is an admissible subspace for $(A,B)$ and that $(A,B)$ satisfies the so-called Fattorini-Hautus test, namely:
\begin{equation}\label{FH test}
\ker(\lambda-A^*) \cap \ker B^*=\ens{0}, \quad \forall \lambda \in \C.
\end{equation}
Assume in addition that there exists $T_0>0$ such that, for every $T>T_0$,	the following two properties hold:
\begin{enumerate}[(i)]
\item\label{cpct hyp 1}
There exist two complex Banach spaces $E_1,E_2$, a compact operator $P: E_1 \longrightarrow E_2$, a linear operator $L: \dom{A^*} \longrightarrow E_1$ and $C>0$ such that, for every $z^1 \in \dom{A^*}$,
\begin{gather}
\norm{z(0)}_H^2
 \leq C\left(\int_0^T \norm{B^*z(t)}_U^2 \, dt+\norm{P L z^1}_{E_2}^2\right), \label{obs ineq plus cpct 1}
\\
\norm{Lz^1}_{E_1}^2 \leq C\left(\norm{z(0)}_H^2+\int_0^T \norm{B^*z(t)}_U^2 \, dt\right), \label{obs ineq plus cpct 2}
\end{gather}
where $z \in C^0([0,T];\dom{A^*})$ is the solution to the adjoint system \eqref{adj syst} (with $\tau=T$).

\item
For every $0<t_1<t_2<T-T_0$, there exists $C>0$ such that, for every $z^1 \in \dom{A^*}$,
\begin{equation}\label{obs ineq plus cpct 3}
\norm{z(t_2)}_H^2\leq C\left(\norm{z(t_1)}_H^2+\int_{t_1}^{t_2} \norm{B^*z(t)}_U^2 \, dt\right),
\end{equation}
where $z \in C^0([0,T];\dom{A^*})$ is the solution to the adjoint system \eqref{adj syst} (with $\tau=T$).

\end{enumerate}
Then, the system $(A,B)$ is null controllable in time $T$ for every $T>T_0$.

\end{theorem}

The proof of this result is postponed to Appendix \ref{sect compactness for NC} for the sake of the presentation.
It is based on arguments developed in the proofs of \cite[Theorem 2]{CN21} and \cite[Lemma 2.6]{DO18} (see also the references therein).
Let us just mention at this point that, in general, the compactness-uniqueness method is designed for the exact controllability property.
It is only thanks to the property \eqref{obs ineq plus cpct 3} that we are able to consider the null controllability property here.

\begin{remark}\label{rem complexification}
In most applications we encounter real systems, that is $H$ and $U$ are real Banach spaces.
To apply what precedes, we have to consider their so-called complexifications as well as the complexifications of the operators $A$ and $B$.
By splitting the complex system (i.e. the system corresponding to these complexifications) into real and imaginary parts, it is not difficult to check that the real system is controllable if, and only if, so is the complex system.
\end{remark}

\subsection{Proof of Theorem \ref{thm remove gmm}}

Let us now show how to use the general result Theorem \ref{thm compactness} in order to obtain Theorem \ref{thm remove gmm}.
We only prove the inequality ``$\leq$'' in \eqref{Tinf Gmm} (which is the most important one), the other inequality can be established similarly.
Let then $T_0>0$ be such that
\begin{equation}\label{syst Gmm zero}
\left(\Lambda,-,Q,
\begin{pmatrix}
0 \\
G_{+-}
\end{pmatrix}
\right)
\end{equation}
is null controllable in time $T_0$ and let us show that necessarily $\Tinf(\Lambda,-,Q,G) \leq T_0$.
This will follow from Theorem \ref{thm compactness} once we will have checked that the system $(\Lambda,-,Q,G)$ satisfies all the assumptions of this result.

First of all, we have to recast the system $(\Lambda,-,Q,G)$ as an abstract evolution system of the form \eqref{abst syst}.
This is quite standard.
To identify what are the operators $A$ and $B$ (in fact, we first find $A^*$ and $B^*$), we repeat the procedure that led to Definition \ref{def sol abst} on the system \eqref{syst gen} (with $M=0$), where taking the adjoints is replaced by an integration by parts in space.
This gives the following.

\begin{itemize}
\item
The state and control spaces are
$$H=L^2(0,1)^n \, (=L^2(0,1;\C^n)), \quad U=\C^m.$$
They are equipped with their usual inner products.

\item
The unbounded linear operator $A: \dom{A} \subset H \longrightarrow H$ is defined for every $y \in \dom{A}$ by
$$
(Ay)(x)=-\Lambda(x) \px{y}(x)+G(x) y_-(0),
\quad x \in (0,1),
$$
with domain
$$\dom{A}
=\ens{y \in H^1(0,1)^n \st y_-(1)=0,\quad y_+(0)=Qy_-(0)}.
$$

\item
It is clear that $\dom{A}$ is dense in $H$ since it contains $C^{\infty}_c(0,1)^n$.
A computation shows that
$$
\dom{A^*}
=\ens{z \in H^1(0,1)^n \st z_-(0)=R^*z_+(0)+\int_0^1 K(\xi)^*z(\xi) \, d\xi, \quad z_+(1)=0},
$$
where $R \in \R^{p \times m}$ and $K \in L^{\infty}(0,1)^{n \times m}$ are defined by
$$
R=-\Lambda_{++}(0)Q\Lambda_{--}(0)^{-1},
\qquad
K(\xi)=-G(\xi)\Lambda_{--}(0)^{-1},
$$
and we have, for every $z \in \dom{A^*}$,
$$
(A^*z)(x)=
\Lambda(x)\px{z}(x)
+\px{\Lambda}(x)z(x),
\quad x \in (0,1).
$$

\item
The control operator $B \in \lin{U,\dom{A^*}'}$ is given for every $u \in U$ and $z \in \dom{A^*}$ by
$$
\ps{Bu}{z}{\dom{A^*}',\dom{A^*}}=\ps{u}{-\Lambda_{--}(1) z_-(1)}{\C^m}.
$$
Note that $B$ is well defined since $Bu$ is continuous on $H^1(0,1)^n$ (by the trace theorem $H^1(0,1)^n \hookrightarrow C^0([0,1])^n$) and since the graph norm $\norm{\cdot}_{\dom{A^*}}$ and $\norm{\cdot}_{H^1(0,1)^n}$ are equivalent norms on $\dom{A^*}$.

\item
Finally, the adjoint $B^* \in \lin{\dom{A^*},U}$ is given for every $z \in \dom{A^*}$ by
$$B^*z=-\Lambda_{--}(1) z_-(1).$$

\end{itemize}

We can prove that $A$ is closed and that both $A,A^*$ are quasi-dissipative, so that $A$ generates a $C_0$-semigroup by a well-known corollary of Lumer-Phillips theorem.

Since the other properties to check depend on the adjoint system, it is convenient to write it explicitly:
$$
\begin{dcases}
\pt{z}(t,x)+\Lambda(x) \px{z}(t,x)=-\px{\Lambda}(x) z(t,x), \\
z_-(t,0)=R^*z_+(t,0)+\int_0^1 K(\xi)^*z(t,\xi) \, d\xi, 
\quad z_+(t,1)=0, \\
z(T,x)=z^1(x).
\end{dcases}
$$

Using the method of characteristics it is easy to prove the estimate \eqref{estim admis} for $\tau \leq T_1(\Lambda)$, which shows that $H$ is an admissible subspace for $(A,B)$.

Therefore, the abstract control system is well posed in $H$.
To rigorously justify that this pair $(A,B)$ is ``the'' abstract form of $(\Lambda,-,Q,G)$ we have to reason in terms of notions of solution:

\begin{proposition}
The solution to system $(\Lambda,-,Q,G)$ in the sense of Definition \ref{def sol char} coincides with the solution to abstract system \eqref{abst syst} in the sense of Definition \ref{def sol abst} corresponding to the pair $(A,B)$ introduced above.
\end{proposition}

\begin{proof}
We argue by approximation.
Let $y^0 \in L^2(0,1)^n$, $u \in L^2(0,T)^m$ be fixed and let $y$ be the corresponding solution to system $(\Lambda,-,Q,G)$ in the sense of Definition \ref{def sol char}.
\begin{itemize}
\item
We take two approximations $(y^{0,k})_k \subset H^1_0(0,1)^n$ and $(u^k)_k \subset H^1_0(0,T)^m$ such that
\begin{equation}\label{approx ID cont}
y^{0,k} \to y^0 \, \text{ in } L^2(0,1)^n,
\quad
u^k \to u \, \text{ in } L^2(0,T)^m.
\end{equation}

Let $y^k$ be the solution corresponding to $y^{0,k}$ and $u^k$ in the sense of Definition \ref{def sol char}.
Since $y^{0,k}$ and $u^k$ obviously satisfy the $C^0$ compatibility conditions
$$
y^{0,k}_-(1)=u^k(0),
\quad
y^{0,k}_+(0)=Qy^{0,k}_-(0),
$$
we can prove that
$$y^k \in C^0([0,T];H^1(0,1)^n) \cap C^0([0,1];H^1(0,T)^n)$$
(for instance, by adapting the fixed point approach of \cite[Appendix A.2]{CHOS21} in the above space -- the regularity $G \in L^{\infty}(0,1)^{n \times m}$ is enough after a suitable change of variable).
In particular, $y^k \in H^1((0,T)\times(0,1))^n$ and it satisfies \eqref{syst gen} almost everywhere (with $M=0$, $y=y^k$, $u=u^k$ and $y^0=y^{0,k}$).

\item
Repeating the procedure that led to Definition \ref{def sol abst} we easily check that $y^k$ is the solution to abstract system \eqref{abst syst} in the sense of Definition \ref{def sol abst}, i.e. it satisfies identity \eqref{def sol} (with $y=y^k$, $y^0=y^{0,k}$ and $u=u^k$).
Using \eqref{approx ID cont} and
$$
y^k \to y \, \text{ in } C^0([0,T];L^2(0,1)^n),
$$
(this follows from \eqref{estim well posed} and \eqref{approx ID cont}), we can pass to the limit $k \to +\infty$ in this identity to obtain that $y$ is the solution to abstract system \eqref{abst syst} in the sense of Definition \ref{def sol abst}.
\end{itemize}
\end{proof}

We will now check that our pair $(A,B)$ satisfies the assumptions of Theorem \ref{thm compactness}.

\begin{itemize}
\item
The Fattorini-Hautus test \eqref{FH test} is easy to check.
Indeed, if $\lambda \in \C$ and $z \in \dom{A^*}$ are such that $A^*z=\lambda z$ and $B^*z=0$, then in particular $z \in H^1(0,1)^n$ solves the system of linear ODEs
$$
\begin{dcases}
\px{z}(x)=\Lambda(x)^{-1}\left(-\px{\Lambda}(x)+\lambda \Id_{\R^{n \times n}}\right)z(x), \quad x \in (0,1), \\
z(1)=0,
\end{dcases}
$$
so that $z=0$ by uniqueness.
\end{itemize}

Below, $C$ denotes a positive number that may change from line to line but that never depends on $z^1$ or $t$.

\begin{itemize}
\item
The inequality \eqref{obs ineq plus cpct 3} is also not difficult to check.
Indeed, for $0<t_1<t_2<T-T_0$, using the method of characteristics, we have
$$
\norm{z_-(t_2,\cdot)}_{L^2(0,1)^m}^2 \leq C\left(
\norm{z_-(t_1,\cdot)}_{L^2(0,1)^m}^2
+\int_{t_1}^{t_2} \norm{z_-(t,1)}^2_{\C^m} \, dt
\right),
$$
and, provided that $T_0 \geq T_{m+1}(\Lambda)$ and using that $z_+(\cdot,1)=0$, we also have
$$
\norm{z_+(t_2,\cdot)}_{L^2(0,1)^p}^2 \leq C
\norm{z_+(t_1,\cdot)}_{L^2(0,1)^p}^2.
$$
We recall that, since the system \eqref{syst Gmm zero} is null controllable in time $T_0$ by assumption, we necessarily have $T_0 \geq T_{m+1}(\Lambda)$ (see the first step of the proof of Theorem \ref{thm NC}).

\item
Let us now investigate the estimate \eqref{obs ineq plus cpct 1}.
Let $T>T_0$.
We will prove that there exists $H \in L^{\infty}((0,T)\times(0,T))^{m \times m}$ such that, for every $z^1 \in L^2(0,1)^n$,
\begin{equation}\label{obs ineq plus cpct 1 concrete}
\norm{z(0,\cdot)}_{L^2(0,1)^n}^2 \leq C\left(
\int_0^T \norm{z_-(t,1)}_{\C^m}^2 \, dt
+\int_0^T \norm{\int_0^T H(t,s)z_-(s,0) \, ds}_{\C^m}^2 \, dt
\right).
\end{equation}

Let us first make some preliminary observations.
We denote by $\zeta$ the solution to the adjoint system of \eqref{syst Gmm zero} in $(0,T)$ with final data $z^1$, and we set
\begin{equation}\label{def theta}
\theta=z-\zeta.
\end{equation}
Clearly, it satisfies
$$
\begin{dcases}
\pt{\theta}(t,x)+\Lambda(x) \px{\theta}(t,x)=-\px{\Lambda}(x) \theta(t,x),
\\
\theta_-(t,0)=R^*\theta_+(t,0)
+\int_0^1 K_{+-}(\xi)^*\theta_+(t,\xi) \, d\xi
+\int_0^1 K_{--}(\xi)^*z_-(t,\xi) \, d\xi,
\quad \theta_+(t,1)=0,
\\
\theta(T,x)=0.
\end{dcases}
$$
Using the method of characteristics, we immediately see that
\begin{equation}\label{theta plus is zero}
\theta_+=0.
\end{equation}
Consequently, $\theta_-$ solves
$$
\begin{dcases}
\pt{\theta_-}(t,x)+\Lambda_{--}(x) \px{\theta_-}(t,x)=-\px{\Lambda_{--}}(x) \theta_-(t,x),
\\
\theta_-(t,0)
=
\int_0^1 K_{--}(\xi)^*z_-(t,\xi) \, d\xi,
\\
\theta_-(T,x)=0.
\end{dcases}
$$
Since $T>T_0 \geq T_m(\Lambda)$, using the method of characteristics, it is not difficult to see that, for $t \in (0,T)$, we have 
\begin{align}
\norm{\theta_-(t,0)}_{\C^m}^2
&= \norm{\int_0^1 K_{--}(\xi)^*z_-(t,\xi) \, d\xi}_{\C^m}^2 \nonumber
\\
&\leq
C\left(
\norm{\int_t^T H(t,s)z_-(s,0) \, ds}_{\C^m}^2
+\int_0^t \norm{z_-(s,1) }_{\C^m}^2 \, ds
\right), \label{estimate theta t zero}
\end{align}
for some $H \in L^{\infty}((0,T)\times(0,T))^{m \times m}$ independent of $z^1$.
Let us now prove the desired estimate \eqref{obs ineq plus cpct 1 concrete}.
Since by assumption the system \eqref{syst Gmm zero} is null controllable in time $T_0$, and thus in time $T>T_0$, the solution $\zeta$ to its adjoint system satisfies (see Theorem \ref{thm duality})
$$
\norm{\zeta(0,\cdot)}_{L^2(0,1)^n}^2
\leq C \int_0^T \norm{\zeta_-(t,1)}_{\C^m}^2.
$$
Recalling \eqref{def theta} and \eqref{theta plus is zero}, it follows that
$$
\begin{array}{rl}
\ds \norm{z(0,\cdot)}_{L^2(0,1)^n}^2
&\ds \leq 2\norm{\theta(0,\cdot)}_{L^2(0,1)^n}^2
+2\norm{\zeta(0,\cdot)}_{L^2(0,1)^n}^2
\\
&\ds =2\norm{\theta_-(0,\cdot)}_{L^2(0,1)^m}^2
+2\norm{\zeta(0,\cdot)}_{L^2(0,1)^n}^2
\\
&\ds \leq 2\norm{\theta_-(0,\cdot)}_{L^2(0,1)^m}^2
+C\int_0^T \norm{\zeta_-(t,1)}_{\C^m}^2
\\
&\ds \leq 2\norm{\theta_-(0,\cdot)}_{L^2(0,1)^m}^2
+2C\int_0^T \norm{\theta_-(t,1)}_{\C^m}^2
+2C\int_0^T \norm{z_-(t,1)}_{\C^m}^2.
\end{array}
$$

On the other hand, using the method of characteristics and the condition $\theta_-(T,\cdot)=0$, we have
$$
\norm{\theta_-(0,\cdot)}_{L^2(0,1)^m}^2
+\int_0^T \norm{\theta_-(t,1)}_{\C^m}^2
\leq
C\int_0^T \norm{\theta_-(t,0)}_{\C^m}^2 \, dt.
$$

Combined with \eqref{estimate theta t zero} this leads to the desired estimate \eqref{obs ineq plus cpct 1 concrete}.

\item
The estimate \eqref{obs ineq plus cpct 1 concrete} suggests to consider the linear operators
$$P:L^2(0,T)^m \longrightarrow L^2(0,T)^m,
\quad
L: \dom{A^*} \longrightarrow L^2(0,T)^m,
$$
defined by
$$
(Pv)(t)=\int_0^T H(t,s)v(s) \, ds,
\quad
(Lz^1)(s)=z_-(s,0).
$$
From the previous point, \eqref{obs ineq plus cpct 1} is fulfilled.
It is also well-known that operators of the form of $P$ are compact.
Finally, we easily check with the method of characteristics that $L$ satisfies the remaining estimate \eqref{obs ineq plus cpct 2}.
This concludes the proof of Theorem \ref{thm remove gmm}.

\end{itemize}

\section{Largest value of the minimal null control time}\label{sect caract supT}

In this last section we will finally prove the second part of Theorem \ref{main thm}.

\subsection{Characterization for systems $(\Lambda,-,Q,G)$}

We start with systems of the form $(\Lambda,-,Q,G)$, we will deal with the initial system $(\Lambda,M,Q,-)$ in the next section.

\begin{theorem}\label{thm SC}
Let $Q^0 \in \R^{p \times m}$ be in canonical form and let $G \in L^{\infty}(0,1)^{n \times m}$ with $G_{--}=0$ and $G_{+-} \in \mathcal{C}(Q^0)$.
\begin{enumerate}[(i)]
\item\label{item upper bound}
The system $(\Lambda,-,Q^0,G)$ is null controllable in time $T$ for every
\begin{equation}\label{cond T supT}
T \geq \max\ens{
\max_{k \in \ens{1,\ldots,\rho_0}}
T_{m+k}(\Lambda)+T_{c_k}(\Lambda),
\quad
T_{m+\rho_0+1}(\Lambda)+T_m(\Lambda)
},
\end{equation}
where we recall that $\rho_0$ is defined in the statement of Theorem \ref{main thm}.

\item\label{upper bound is reached}
Assume that the condition \eqref{CNS Tinf indep M} fails and let $G \in \R^{n \times m}$ be the constant matrix whose entries are all equal to zero except for
$$g_{m+\rho_0+1, m}=1.$$
If the corresponding system $(\Lambda,-,Q^0,G)$ is null controllable in time $T$, then $T$ has to satisfy the condition \eqref{cond T supT}.
\end{enumerate}
\end{theorem}

As for Theorem \ref{main thm}, we use the convention that the undefined quantities are simply not taken into account, which more precisely gives:
\begin{itemize}
\item
If $\rho_0=0$, then the condition \eqref{cond T supT} is $T \geq T_{m+1}(\Lambda)+T_m(\Lambda)$.

\item
If $\rho_0=p$, then the condition \eqref{cond T supT} is $T \geq \max\ens{\max_{k \in \ens{1,\ldots,p}} T_{m+k}(\Lambda)+T_{c_k}(\Lambda), \quad T_m(\Lambda)}$.
\end{itemize}

In the second part of the statement we only discussed the case when \eqref{CNS Tinf indep M} fails since otherwise the time on the right-hand side of the inequality in \eqref{cond T supT} coincides with the time on the right-hand side of the inequality in \eqref{cond T infT} and it follows from item \ref{thm NC i1} of Theorem \ref{thm NC} that item \ref{item upper bound} of Theorem \ref{thm SC} then becomes a necessary and sufficient condition.

This result shows in particular that the largest value that $\Tinf(\Lambda,-,Q^0,G)$ can take with respect to $G \in L^{\infty}(0,1)^{n \times m}$ when $G_{--}=0$ and $G_{+-} \in \mathcal{C}(Q^0)$ is equal to the quantity on the right-hand side of the inequality in \eqref{cond T supT}.
This can be extended to arbitrary boundary coupling matrices and arbitrary $G \in L^{\infty}(0,1)^{n \times m}$ thanks to Proposition \ref{prop reduc Q}, Proposition \ref{prop Gtilde} and Theorem \ref{thm remove gmm}.

\begin{proof}[Proof of Theorem \ref{thm SC}]

\begin{enumerate}[1)]
\item
We begin with the proof of the first item.
Let first $i \in \ens{1,\ldots,m}$ be fixed.
Since $T \geq T_i(\Lambda)$, which means that $\ssin_i(T,x)>0$ for every $x \in (0,1)$ as we have seen in the first step of the proof of Theorem \ref{thm NC}, and since $G_{--}=0$ by assumption, the null controllability condition $y_i(T,\cdot)=0$ is equivalent to (see \eqref{sol char def 1} and \eqref{def bi i<m})
$$u_i\left(\ssin_i(T,\cdot)\right)=0 \quad \text{ in } (0,1).$$
Since $i \leq m$, the map $x \mapsto \ssin_i(T,x)$ is non decreasing (see \eqref{ssin monotonicity}) with $\ssin_i(T,1)=T$.
Thus, the previous condition is also equivalent to
\begin{equation}\label{ui zero}
u_i=0 \quad \text{ in } \left(\ssin_i(T,0), T\right).
\end{equation}

\item
Let us now consider $i \in \ens{m+1,\ldots,n}$.
Since $T \geq T_i(\Lambda)$, the null controllability condition $y_i(T,x)=0$ is equivalent to (see \eqref{sol char def 1} and \eqref{def bi i>m})
$$
a_i(x)+b_i(x)
=0,
$$
where
$$
a_i(x)=
\sum_{\substack{j=1 \\ j \not\in \ens{c_1,\ldots,c_{\rho_0}}}}^m
\left(
q^0_{i-m,j} y_j\left(\ssin_i(T,x),0\right)
+\int_{\ssin_i(T,x)}^{T}  g_{ij}\left(\chi_i(s;T,x)\right) y_j(s,0) \, ds
\right),
$$
and
$$
b_i(x)=
\sum_{k=1}^{\rho_0} 
\left(
q^0_{i-m,c_{k}} y_{c_{k}}\left(\ssin_i(T,x),0\right)
+\int_{\ssin_i(T,x)}^{T} g_{i c_{k}}\left(\chi_i(s;T,x)\right) y_{c_{k}}(s,0) \, ds
\right).
$$

\begin{itemize}
\item
We first consider the case $i \geq m+\rho_0+1$ (which happens only if $\rho_0<p$).
Clearly, we have $b_i=0$ in that situation since $Q^0$ is in canonical form, $G_{+-} \in \mathcal{C}(Q^0)$ and \eqref{r less than rhozero}.
Let us show that we can choose $u_j$ for $j \not\in \ens{c_1,\ldots,c_{\rho_0}}$ so that $a_i=0$ as well.
Since $x \mapsto \ssin_i(T,x)$ is non increasing for $i \geq m+1$ (recall \eqref{ssin monotonicity}), it is sufficient to choose it such that
\begin{equation}\label{cond trace yj}
y_j(\cdot,0)=0 \quad \text{ in } \left(\ssin_i(T,1), T\right).
\end{equation}
Since $T \geq T_{m+\rho_0+1}(\Lambda)+T_m(\Lambda)$ by assumption, we have in particular $T \geq T_i(\Lambda)+T_j(\Lambda)$ for the indices $i,j$ considered (recall \eqref{order times}).
This condition can be written as $T \geq \ssout_i(\ssout_j(0,1), 0)$ (see \eqref{explicit formula ssin} and \eqref{def Ti gen}) or, equivalently (see \eqref{inv sin}),
$$\ssin_j(\ssin_i(T,1), 0) \geq 0.$$
Since $s \mapsto \ssin_j(s,0)$ is increasing (see \eqref{ssin monotonicity}), this is equivalent to
$$\ssin_j(s,0)>0, \quad \forall s \in \left(\ssin_i(T,1), T\right).$$
As a result, we see that \eqref{cond trace yj} holds if, and only if, (see \eqref{sol char def 1}, \eqref{def bi i<m} and recall that $G_{--}=0$)
$$u_j(\ssin_j(\cdot,0))=0 \quad \text{ in } \left(\ssin_i(T,1), T\right).$$
Using again that $s \mapsto \ssin_j(s,0)$ is increasing, this means that
$$u_j=0 \quad \text{ in } \left(\ssin_j(\ssin_i(T,1), 0), \quad \ssin_j(T,0)\right).$$
Observe that this is compatible with \eqref{ui zero} since these two intervals are disjoint.

\item
Let us now consider the case $i \leq m+\rho_0$ (which happens only if $\rho_0 \neq 0$).
Since $Q^0$ is in canonical form, $G_{+-} \in \mathcal{C}(Q^0)$ and \eqref{r less than rhozero}, we see that $a_i(x)+b_i(x)=0$ is equivalent to
\begin{equation}\label{case i<m+rhozero}
a_i(x)+
y_{c_{i-m}}\left(\ssin_i(T,x),0\right)
+\sum_{k=i-m+1}^{\rho_0} \int_{\ssin_i(T,x)}^T g_{i c_{k}}\left(\chi_i(s;T,x)\right) y_{c_{k}}(s,0) \, ds
=0.
\end{equation}
Let us show that we can choose $u_{c_1},\ldots, u_{c_{\rho_0}}$ so that this identity is satisfied.
By assumption, we have $T \geq T_i(\Lambda)+T_{c_{i-m}}(\Lambda)$ for every $i \in \ens{m+1,\ldots,m+\rho_0}$.
As in the previous point we can check that this condition can be written as
$$\ssin_{c_{i-m}}(\ssin_i(T,1), 0) \geq 0.$$
Since $x \mapsto \ssin_{c_{i-m}}(\ssin_i(T,x), 0)$ is decreasing (see \eqref{ssin monotonicity}), this is equivalent to
$$\ssin_{c_{i-m}}(\ssin_i(T,x),0)>0, \quad \forall x \in (0,1).$$
As a result, we see that \eqref{case i<m+rhozero} holds if, and only if, (see \eqref{sol char def 1}, \eqref{def bi i<m} and recall that $G_{--}=0$)
$$
u_{c_{i-m}}\left(\ssin_{c_{i-m}}(\ssin_i(T,x),0)\right)
=-a_i(x)
-\sum_{k=i-m+1}^{\rho_0} \int_{\ssin_i(T,x)}^T g_{i c_{k}}\left(\chi_i(s;T,x)\right) y_{c_{k}}(s,0) \, ds.
$$
Since $a_i$ is known (it only concerns $u_j$ for $j \not\in \ens{c_1,\ldots,c_{\rho_0}}$), we see by induction (starting with $i=m+\rho_0$) that this formula determines the values of $u_{c_{i-m}}$ in the interval
$$
\left(\ssin_{c_{i-m}}(\ssin_i(T,1),0), \quad \ssin_{c_{i-m}}(T,0)\right)
$$
(the map $x \mapsto \ssin_{c_{i-m}}(\ssin_i(T,x),0)$ is non increasing and $\ssin_i(T,0)=T$).
Observe once again that this is compatible with \eqref{ui zero} since these two intervals are disjoint.
\end{itemize}
This concludes the proof of the first item \ref{item upper bound} of Theorem \ref{thm SC}.

\item
Let us now prove item \ref{upper bound is reached} of Theorem \ref{thm SC}.
Assume that the condition \eqref{CNS Tinf indep M} fails, let $G$ be the constant matrix introduced in the statement, and assume that the corresponding system $(\Lambda,-,Q^0,G)$ is null controllable in time $T$.
Since \eqref{CNS Tinf indep M} fails, the condition \eqref{cond T supT} is simply
\begin{equation}\label{cond time 22 subsyst}
T \geq T_{m+\rho_0+1}(\Lambda)+T_m(\Lambda).
\end{equation}
Since the system $(\Lambda,-,Q^0,G)$ is null controllable in time $T$ by assumption, the following $2 \times 2$ subsystem also has to be null controllable in time $T$:
$$
\begin{dcases}
\pt{y_m} (t,x)+\lambda_m(x) \px{y_m}(t,x)=0, \\
\pt{y_{m+\rho_0+1}}(t,x)+\lambda_{m+\rho_0+1}(x) \px{y_{m+\rho_0+1}}(t,x)= y_m(t,0), \\
y_m(t,1)=u_m(t), \quad y_{m+\rho_0+1}(t,0)=q_{\rho_0+1,m}^0 y_m(t,0).
\end{dcases}
$$
Let us show that, whether $q_{\rho_0+1,m}^0=1$ or $q_{\rho_0+1,m}^0=0$, we necessarily have \eqref{cond time 22 subsyst}.
If $q_{\rho_0+1,m}^0=1$, then this follows from item \ref{thm NC i1} of Theorem \ref{thm NC}.
Let us then consider the case $q_{\rho_0+1,m}^0=0$.
As before, it is clearly necessary that $T \geq T_{m+\rho_0+1}(\Lambda)$ and, under this condition, the null controllability condition $y_{m+\rho_0+1}(T,x)=0$ becomes equivalent to (see \eqref{sol char def 1} and \eqref{def bi i>m})
$$\int_{\ssin_{m+\rho_0+1}(T,x)}^T y_m(s,0) \, ds=0.$$
Using the change of variable $\xi \mapsto s=\ssin_{m+\rho_0+1}(T,\xi)$, this holds if, and only if,
$$\int_0^x y_m(\ssin_{m+\rho_0+1}(T,\xi),0) \frac{\partial \ssin_{m+\rho_0+1}}{\partial \xi}(T,\xi) \, d\xi=0.$$
Taking the derivative with respect to $x$, this is also equivalent to
$$y_m(\ssin_{m+\rho_0+1}(T,\cdot),0)=0 \quad \text{ in } (0,1).$$
It is now not difficult to see that we can choose $u_m$ such that this condition holds if, and only if, we have \eqref{cond time 22 subsyst}.

\end{enumerate}

\end{proof}

\subsection{Proof of the second part of Theorem \ref{main thm}}

Let us now show how to combine all the previous results in order to obtain the desired characterization of the largest minimal null control time for the initial system $(\Lambda,M,Q,-)$.

\begin{proof}[Proof of item \ref{item supT} of Theorem \ref{main thm}]
Let $Q \in \R^{p \times m}$ be fixed.
\begin{enumerate}[1)]
\item
\begin{itemize}
\item
By item \ref{item upper bound} of Theorem \ref{thm SC}, we have
$$
\Tinf(\Lambda,-,Q^0,G) \leq
\max\ens{
\max_{k \in \ens{1,\ldots,\rho_0}}
T_{m+k}(\Lambda)+T_{c_k}(\Lambda),
\quad
T_{m+\rho_0+1}(\Lambda)+T_m(\Lambda)
},
$$
for every $G \in L^{\infty}(0,1) ^{n \times m}$ with $G_{--}=0$ and $G_{+-} \in \mathcal{C}(Q^0)$, where $Q^0$ is the canonical form of $Q$.

\item
By Theorem \ref{thm remove gmm}, this inequality remains true for every $G \in L^{\infty}(0,1)^{n \times m}$ with $G_{+-} \in \mathcal{C}(Q^0)$.

\item
By Proposition \ref{prop Gtilde}, this inequality remains true for every $G \in L^{\infty}(0,1)^{n \times m}$.

\item
By Proposition \ref{prop reduc Q}, this inequality remains true by changing $Q^0$ into $Q$.

\item
By Proposition \ref{prop remove diag} and Theorem \ref{thm backstepping}, this inequality remains true for the system $(\Lambda,M,Q,-)$ for any $M \in L^{\infty}(0,1)^{n \times n}$.
\end{itemize}
In summary, we have established the following upper bound:
$$
\Tinf(\Lambda,M,Q) \leq
\max\ens{
\max_{k \in \ens{1,\ldots,\rho_0}}
T_{m+k}(\Lambda)+T_{c_k}(\Lambda),
\quad
T_{m+\rho_0+1}(\Lambda)+T_m(\Lambda)
},
$$
valid for every $M \in L^{\infty}(0,1)^{n \times n}$.

\item
Let us now show that this upper bound is reached for some special $M$.
If the condition \eqref{CNS Tinf indep M} is satisfied, then this upper bound coincides with the lower bound, and we know that this latter is reached for $M=0$ (Remark \ref{rem times are reached}).
Let us now assume that the condition \eqref{CNS Tinf indep M} is not satisfied.

\begin{itemize}
\item
Then, we know from Theorem \ref{thm SC} that this upper bound is the minimal null control time of the system $(\Lambda,-,Q^0,G)$ for the constant matrix $G \in \R^{n \times m}$ whose entries are all equal to zero except for
$$g_{m+\rho_0+1, m}=1.$$
Let us now find the corresponding matrix $M$.

\item
We decompose $Q$ in its canonical form: $LQU=Q^0$.
By Proposition \ref{prop reduc Q}, this upper bound is the minimal null control time of the system $(\Lambda,-,Q,\widehat{G})$ for any $\widehat{G}$ such that $\Theta(\widehat{G})=G$.
Now, it follows from the proof of Proposition \ref{prop reduc Q} that, for constant matrices, $\Theta$ is simply given by
$$
\Theta(\widehat{G})=
\begin{pmatrix}
U^{-1}\widehat{G}_{--}U
\\
L\widehat{G}_{+-}U
\end{pmatrix},
\quad \forall \widehat{G} \in \R^{n \times m}.
$$
Therefore, $\widehat{G}$ is the matrix whose entries are all equal to zero except for
$$\hat{g}_{m+i, m}=\ell^{i,\rho_0+1}, \quad \forall i \in \ens{\rho_0+1,\ldots,p},$$
where $L^{-1}=(\ell^{ij})_{1 \leq i,j \leq p}$.

\item
By Theorem \ref{thm backstepping}, this upper bound is the minimal null control time of the system $(\Lambda,M,Q,-)$ for any $M \in \mathcal{M}$ such that $\Gamma_A(M)=\widehat{G}$ for some $A \in \mathcal{F}$.
Let us determine $A$ and $M$ such that this identity holds.
By definition of $\Gamma_A(M)$ (see \eqref{def GammaA}), this is equivalent to
$$
\begin{dcases}
0=-K_{--}(x,0)\Lambda_{--}(0)-K_{-+}(x,0)\Lambda_{++}(0)Q, \\
\widehat{G}_{+-}=-K_{+-}(x,0)\Lambda_{--}(0)-K_{++}(x,0)\Lambda_{++}(0)Q,
\end{dcases}
$$
where $K$ is the solution to the kernel equations \eqref{kern equ} with additional boundary conditions \eqref{additional conditions 1}-\eqref{additional conditions 2} provided by $A$.
Let us rewrite these kernel equations by blocks:
$$
\begin{dcases}
\Lambda_{--}(x)\px{K_{--}}(x,\xi)
+\pxi{K_{--}}(x,\xi)\Lambda_{--}(\xi)
\\
\quad
+K_{--}(x,\xi)\left(\pxi{\Lambda_{--}}(\xi)+M_{--}(\xi)\right)
+K_{-+}(x,\xi)M_{+-}(\xi)
=0,
\\
\Lambda_{--}(x)K_{--}(x,x)-K_{--}(x,x)\Lambda_{--}(x)=M_{--}(x).
\end{dcases}
$$
$$
\begin{dcases}
\Lambda_{--}(x)\px{K_{-+}}(x,\xi)
+\pxi{K_{-+}}(x,\xi)\Lambda_{++}(\xi)
\\
\quad
+K_{--}(x,\xi)M_{-+}(\xi)
+K_{-+}(x,\xi)\left(\pxi{\Lambda_{++}}(\xi)+M_{++}(\xi)\right)
=0,
\\
\Lambda_{--}(x)K_{-+}(x,x)-K_{-+}(x,x)\Lambda_{++}(x)=M_{-+}(x).
\end{dcases}
$$
$$
\begin{dcases}
\Lambda_{++}(x)\px{K_{+-}}(x,\xi)
+\pxi{K_{+-}}(x,\xi)\Lambda_{--}(\xi)
\\
\quad
+K_{+-}(x,\xi)\left(\pxi{\Lambda_{--}}(\xi)+M_{--}(\xi)\right)
+K_{++}(x,\xi)M_{+-}(\xi)
=0,
\\
\Lambda_{++}(x)K_{+-}(x,x)-K_{+-}(x,x)\Lambda_{--}(x)=M_{+-}(x).
\end{dcases}
$$
$$
\begin{dcases}
\Lambda_{++}(x)\px{K_{++}}(x,\xi)
+\pxi{K_{++}}(x,\xi)\Lambda_{++}(\xi)
\\
\quad
+K_{+-}(x,\xi)M_{-+}(\xi)
+K_{++}(x,\xi)\left(\pxi{\Lambda_{++}}(\xi)+M_{++}(\xi)\right)
=0,
\\
\Lambda_{++}(x)K_{++}(x,x)-K_{++}(x,x)\Lambda_{++}(x)=M_{++}(x).
\end{dcases}
$$
Note that the subsystems satisfied by $(K_{--},K_{-+})$ and $(K_{+-},K_{++})$ are not coupled.
By uniqueness of the solution to these equations (see Theorem \ref{thm kern}), we see that
\begin{gather*}
M_{-+}=0 \quad \Longrightarrow \quad K_{-+}=0,
\\
M_{++}=A_{++}=0, \quad M_{-+}=0 \quad \Longrightarrow \quad K_{++}=0,
\\
M_{--}=A_{--}=0, \quad K_{-+}=0 \quad \Longrightarrow \quad K_{--}=0.
\end{gather*}
Therefore, it only remains to determine $M_{+-}$ such that
$$K_{+-}(x,0)=-\widehat{G}_{+-}\Lambda_{--}(0)^{-1}.$$
Let $i \in \ens{m+1,\ldots,n}$ and $j \in \ens{1,\ldots,m}$ be fixed.
The equation for $k_{ij}$ is now simply
\begin{equation}\label{kern equ pm}
\begin{dcases}
\lambda_i(x) \px{k_{ij}}(x,\xi)
+\pxi{k_{ij}}(x,\xi)\lambda_j(\xi)
+k_{ij}(x,\xi)\pxi{\lambda_j}(\xi)
=0,
\\
k_{ij}(x,x)=\frac{m_{ij}(x)}{\lambda_i(x)-\lambda_j(x)}.
\end{dcases}
\end{equation}

Let $s \mapsto \zeta_{ij}(s;x,\xi)$ be the associated characteristic passing through $(x,\xi)$:
$$
\begin{dcases}
\pas{\zeta_{ij}}(s;x,\xi)= \frac{\lambda_j(\zeta_{ij}(s;x,\xi))}{\lambda_i(s)}, \\
\zeta_{ij}(x;x,\xi)=\xi.
\end{dcases}
$$
The solution to \eqref{kern equ pm} is explicit:
$$
k_{ij}(x,\xi)=
\frac{m_{ij}(\ssin_{ij}(x,\xi))}{\lambda_i(\ssin_{ij}(x,\xi))-\lambda_j(\ssin_{ij}(x,\xi))}
\frac{\lambda_j(\ssin_{ij}(x,\xi))}{\lambda_j(\xi)},
$$
where $\ssin_{ij}(x,\xi) \in (0,x)$ is the unique solution to
$$\zeta_{ij}\left(\ssin_{ij}(x,\xi);x,\xi\right)=\ssin_{ij}(x,\xi).$$
Thus, the desired condition $k_{ij}(\cdot,0)=-\hat{g}_{ij}/\lambda_j(0)$ is equivalent to
$$
m_{ij}(\ssin_{ij}(x,0))=
\frac{\lambda_i(\ssin_{ij}(x,0))-\lambda_j(\ssin_{ij}(x,0))}{-\lambda_j(\ssin_{ij}(x,0))}
\hat{g}_{ij},
\quad x \in (0,1).
$$

\end{itemize}

\end{enumerate}
\end{proof}

\section*{Acknowledgements}

This project was supported by National Natural Science Foundation of China (Nos. 12122110 and 12071258), the Young Scholars Program of Shandong University (No. 2016WLJH52) and National Science Centre, Poland UMO-2020/39/D/ST1/01136.
For the purpose of Open Access, the authors have applied a CC-BY public copyright licence to any Author Accepted Manuscript (AAM) version arising from this submission.

\appendix

\section{An example of non equivalent hyperbolic systems}\label{sect counterexample}

In this appendix we present an explicit example of hyperbolic systems which are not equivalent in the sense of Definition \ref{def syst equiv}.
This example is important to illustrate that, in general, it is not possible to obtain a simpler system than the one we obtained in the present article if we only use invertible transformations (see Remark \ref{rem canonical form}).
It also motivates the use of the compactness-uniqueness method to establish the important result Theorem \ref{thm remove gmm}.
We refer to \cite[Section 4.3]{CN19} for a close but different example.

We consider the following simple $3 \times 3$ systems with constant coefficients:
\begin{equation}\label{counterex g--}
\begin{dcases}
\pt{y_1}(t,x)-\px{y_1}(t,x)=0, \\
\pt{y_2}(t,x)-\frac{1}{2}\px{y_2}(t,x)=a y_1(t,0), \\
\pt{y_3}(t,x)+\px{y_3}(t,x)=by_2(t,0),
\end{dcases}
\end{equation}
where $a,b \in \R$ are some parameters, and with boundary conditions
\begin{equation}\label{counterex g-- BC}
\begin{dcases}
y_1(t,1)=u_1(t), \\
y_2(t,1)=u_2(t),
\end{dcases}
\quad y_3(t,0)=y_1(t,0).
\end{equation}
We are in the case $m=2$, $p=1$ and the matrices $\Lambda,G$ and $Q$ are
$$
\Lambda=
\left(\begin{array}{cc|c}
-1 & 0 & 0 \\
0 & -1/2 & 0 \\
\hline
0 & 0 & 1
\end{array}\right)
, \quad
G=G_{ab}=
\left(\begin{array}{cc}
0 & 0 \\
a & 0 \\
\hline
0 & b
\end{array}\right)
, \quad
Q=
\left(\begin{array}{cc}
1 & 0
\end{array}\right)
.
$$
Note as well that we are in an ideal configuration:
\begin{itemize}
\item
$Q$ is in canonical form.

\item
$(G_{ab})_{--}$ is strictly lower triangular.

\item
$(G_{ab})_{+-} \in \mathcal{C}(Q)$.
\end{itemize}

Clearly, $\rho=\rho_0=1$ and $(r_1,c_1)=(1,1)$.
It follows from the results of the present article (actually, a direct proof is also possible) that the minimal null control time of the system \eqref{counterex g--}-\eqref{counterex g-- BC} is
$$\Tinf(\Lambda,-,Q,G_{ab})=2, \quad \forall a,b \in \R.$$
In particular, the system \eqref{counterex g--}-\eqref{counterex g-- BC} is null controllable in time $T$ for every $T>2$.
Let us now study the null controllability properties of this system in this critical time:

\begin{proposition}\label{prop counterex}
The system \eqref{counterex g--}-\eqref{counterex g-- BC} is null controllable in time $T=2$ if, and only if,
\begin{equation}\label{cns a}
ab \not\in \critset=\ens{-\left(\frac{\pi}{2}+k\pi\right)^2 \st k \in \N}.
\end{equation}
\end{proposition}

\begin{remark}
It follows from this result and Proposition \ref{basic prop equiv systs} that
$$
(\Lambda,-,Q,G_{ab})
\quad \text{ is not equivalent to } \quad
(\Lambda,-,Q,G_{cd}), \quad \text{ if } ab \in \critset, \, cd \not\in \critset.
$$
In particular, it is not possible to transform the system $(\Lambda,-,Q,G_{ab})$ into $(\Lambda,-,Q,G_{0b})$ or $(\Lambda,-,Q,G_{a0})$ when $ab \in \critset$ (in other words, we cannot remove $(G_{ab})_{--}$ nor $(G_{ab})_{+-}$ in this case).
\end{remark}

\begin{proof}[Proof of Proposition \ref{prop counterex}]
The solution to the system \eqref{counterex g--}-\eqref{counterex g-- BC} is explicit (see Section \ref{sect sol char}):
\begin{gather*}
y_1(t,x)=
\begin{dcases}
u_1(t-1+x) & \text{ if } t-1+x>0, \\
y_1^0(t+x) & \text{ if } t-1+x<0,
\end{dcases}
\\
y_2(t,x)=
\begin{dcases}
u_2(t-2(1-x))
+a \int_{t-2(1-x)}^t y_1(s,0) \, ds
& \text{ if } t-2(1-x)>0, \\
y_2^0\left(\frac{t}{2}+x\right)
+a \int_0^t y_1(s,0) \, ds
& \text{ if } t-2(1-x)<0,
\end{dcases}
\\
y_3(t,x)=
\begin{dcases}
y_1(t-x,0)
+b\int_{t-x}^t y_2(s,0) \, ds
& \text{ if } t-x>0, \\
y_3^0(-t+x)
+ b\int_0^t y_2(s,0) \, ds
& \text{ if } t-x<0.
\end{dcases}
\end{gather*}
Clearly, the null controllability condition $y_1(2,\cdot)=0$ is satisfied if, and only if,
$$u_1=0 \quad \text{ in } (1,2).$$
Similarly, the null controllability condition $y_2(2,\cdot)=0$ holds if, and only if,
$$u_2(t)=-a\int_t^2 y_1(s,0) \, ds, \quad t \in (0,2).$$
Thus, the control $u_2$ is uniquely determined once the values of the control $u_1$ in $(0,1)$ are known.
The remaining condition $y_3(2,x)=0$ is equivalent to
$$y_1(2-x,0)+b\int_{2-x}^2 y_2(s,0) \, ds=0,$$
and thus to
$$
u_1(1-x)
+b\int_{2-x}^2 y_2^0\left(\frac{s}{2}\right) \, ds
+ab x \int_0^1 y_1^0(\theta) \, d\theta
+ab \int_{2-x}^2 \int_1^s u_1(\theta-1) \, d\theta ds
=0.
$$
Using the change of variables $t=1-x$ and $\sigma=\theta-1$, this is also equivalent to
\begin{equation}\label{int equ for uone}
u_1(t)+ab \int_t^1 \int_0^s u_1(\sigma) \, d\sigma ds=f(t), \quad t \in (0,1),
\end{equation}
where we introduced the following function depending only on the initial data:
\begin{equation}\label{def f counterex}
f(t)=
-b \int_{1+t}^2  y_2^0\left(\frac{s}{2}\right) \, ds
-ab (1-t) \int_0^1 y_1^0(\theta) \, d\theta.
\end{equation}
This identity can be rewritten as
\begin{equation}\label{compact form}
(\Id-K) \begin{pmatrix} u_1 \\ \int_0^{\cdot} u_1(\sigma) \, d\sigma \end{pmatrix}
=\begin{pmatrix} f \\ 0 \end{pmatrix},
\end{equation}
where $K:L^2(0,1)^2 \longrightarrow L^2(0,1)^2$ is the operator defined by
$$
\left(K\begin{pmatrix} \alpha \\ \beta \end{pmatrix}\right)(t)
=\begin{pmatrix} -ab \int_t^1 \beta(s) \, ds \\ \int_0^t \alpha(s) \, ds \end{pmatrix}.
$$
Since $K$ is compact, the Fredholm alternative says that \eqref{compact form} has a solution if, and only if,
\begin{equation}\label{cns exist sol}
\begin{pmatrix} f \\ 0 \end{pmatrix} \in (\ker(\Id-K^*))^\perp.
\end{equation}
A simple computation shows that
$$
\left(K^*\begin{pmatrix} \tilde{\alpha} \\ \tilde{\beta} \end{pmatrix}\right)(s)
=\begin{pmatrix} \int_s^1 \tilde{\beta}(t) \, dt \\ -ab \int_0^s \tilde{\alpha}(t) \, dt \end{pmatrix}.
$$
It follows that $\begin{pmatrix} \tilde{\alpha} \\ \tilde{\beta} \end{pmatrix}
\in \ker(\Id-K^*)$ if, and only if, $\tilde{\beta}(s)=-ab \int_0^s \tilde{\alpha}(t) \, dt$ and $\tilde{\alpha}$ solves the following second order linear ODE:
$$
\begin{dcases}
\tilde{\alpha}''(s)-ab \tilde{\alpha}(s)=0, \\
\tilde{\alpha}'(0)=0, \\
\tilde{\alpha}(1)=0.
\end{dcases}
$$
We can check that this ODE has a nonzero solution if, and only if,
$$ab \in \critset,$$
where $\critset$ is the set introduced in \eqref{cns a}.
It follows that we have two possibilities:
\begin{itemize}
\item
If $ab \not\in \critset$, then $\ker(\Id-K^*)=\ens{0}$ and \eqref{compact form} has a (unique) solution $u_1$.
This shows that the system \eqref{counterex g--}-\eqref{counterex g-- BC} is null controllable in time $T=2$.

\item
If $ab \in \critset$, then there exists a nonzero $\begin{pmatrix} \tilde{\alpha} \\ \tilde{\beta} \end{pmatrix}
\in \ker(\Id-K^*)$.
Necessarily, $\tilde{\alpha} \neq 0$ and thus
$$\exists f \in C^{\infty}_c(0,1), \quad \ps{\tilde{\alpha}}{f}{L^2(0,1)} \neq 0.$$
It is clear that we can construct $y^0_2$ and $y^0_1$ that satisfy \eqref{def f counterex} for this $f$ (take for instance $y_1^0=0$ and $y_2^0(x)=f'(2x-1)/b$ for $x \in [1/2,1]$ and $y_2^0(x)=0$ otherwise, note that $b \neq 0$ in the case considered).
For such a $f$, the condition \eqref{cns exist sol} fails and thus there is no corresponding solution $u_1$ to \eqref{int equ for uone}, meaning that the system \eqref{counterex g--}-\eqref{counterex g-- BC} is not null controllable in time $T=2$.
\end{itemize}
\end{proof}

\begin{remark}
We have seen during the proof that, when $ab \not\in \critset$, the control that brings the solution to zero in the critical time $T=2$ is unique (it can also be written explicitly).
\end{remark}

\section{Proof of the abstract compactness-uniqueness result}\label{sect compactness for NC}

The goal of this appendix is to give a proof of Theorem \ref{thm compactness}.
It is inspired from the proofs of \cite[Theorem 2]{CN21} and \cite[Lemma 2.6]{DO18} (see also the references therein).


Here and in what follows, it will be more convenient to work with the expression $S(t)^*z^1$ rather than $z(t)=S(T-t)^*z^1$.
The corresponding assumptions \eqref{obs ineq plus cpct 1}, \eqref{obs ineq plus cpct 2} and \eqref{obs ineq plus cpct 3} become:
\begin{gather}
\norm{S(T)^*z^1}_H^2
 \leq C\left(\int_0^T \norm{B^*S(t)^*z^1}_U^2 \, dt+\norm{P L z^1}_{E_2}^2\right), \label{obs ineq plus cpct 1 bis}
\\
\norm{Lz^1}_{E_1}^2 \leq C\left(\norm{S(T)^*z^1}_H^2+\int_0^T \norm{B^*S(t)^*z^1}_U^2 \, dt\right), \label{obs ineq plus cpct 2 bis}
\\
\norm{S(T-t_2)^*z^1}_H^2\leq C\left(\norm{S(T-t_1)^*z^1}_H^2+\int_{T-t_2}^{T-t_1} \norm{B^*S(t)^*z^1}_U^2 \, dt\right).
\label{obs ineq plus cpct 3 bis}
\end{gather}

\begin{enumerate}[1)]
\item\label{step 1}
Let $T>T_0$ be fixed.
By duality (see Theorem \ref{thm duality}), we have to prove that there exists $C>0$ such that, for every $z^1 \in \dom{A^*}$,
\begin{equation}\label{obs ineq abst bis}
\norm{S(T)^*z^1}_H^2 \leq C\int_0^T \norm{B^*S(t)^*z^1}_U^2 \, dt.
\end{equation}
We argue by contradiction and assume that the observability inequality \eqref{obs ineq abst bis} does not hold.
Then, there exists a sequence $(z^1_n)_{n \geq 1} \subset \dom{A^*}$ such that, for every $n \geq 1$,
$$\norm{S(T)^*z^1_n}_H^2>n \int_0^T \norm{B^*S(t)^*z^1_n}_U^2 \, dt.$$
In particular $S(T)^*z^1_n \neq 0$ and we can normalize $z^1_n$, still denoted by the same, in such a way that
$$
\norm{S(T)^*z^1_n}_H=1,
\quad
\int_0^T \norm{B^*S(t)^*z^1_n}_U^2 \, dt \xrightarrow[n \to +\infty]{} 0.
$$
Using the estimate \eqref{obs ineq plus cpct 2 bis} we obtain that
$$(Lz^1_n)_{n \geq 1} \text{ is bounded in } E_1.$$
Since $P$ is compact, we can extract a subsequence, still denoted by $(z^1_n)_{n \geq 1}$, such that
$$(PLz^1_n)_{n \geq 1} \text{ converges in } E_2.$$
Using now the estimate \eqref{obs ineq plus cpct 1 bis}, we obtain that $(S(T)^*z^1_n)_{n \geq 1}$ is a Cauchy sequence in $H$, and thus converges: there exists $f \in H$ such that
$$S(T)^*z^1_n \xrightarrow[n \to +\infty]{} f \text{ in } H.$$
Besides, $f \neq 0$ since $\norm{f}_H=1$.
In other words, we have shown that
\begin{equation}\label{NT not zero}
N_T \neq \ens{0},
\end{equation}
where $N_{\tau}$ is the subspace defined for every $\tau>0$ by
$$N_{\tau}=\ens{f \in H \st \exists (z^1_n)_{n \geq 1} \subset \dom{A^*}, \quad
\begin{array}{l}
S(\tau)^*z^1_n \xrightarrow[n \to +\infty]{} f \text{ in } H,
\\
B^*S(\cdot)^*z^1_n \xrightarrow[n \to +\infty]{} 0 \text{ in } L^2(0,\tau;U)
\end{array}
}.$$
Let us now study the properties of these subspaces.

\item
First of all, it forms a non-increasing sequence of subspaces:
\begin{equation}\label{NT decreases}
N_{\tau_2} \subset N_{\tau_1}, \quad \forall \tau_2 \geq \tau_1>0.
\end{equation}
Indeed, if $f \in N_{\tau_2}$ and $(z^1_n)_{n \geq 1} \subset \dom{A^*}$ denotes an associated sequence, then we easily check that $f \in N_{\tau_1}$ by considering the sequence $(S(\tau_2-\tau_1)^*z^1_n)_{n \geq 1}$.

\item
Let us now show that
\begin{equation}\label{NT finite dim}
\dim N_{\tau}<+\infty, \quad \forall \tau>T_0.
\end{equation}
By Riesz theorem, it is equivalent to show that the closed unit ball of $N_{\tau}$ is compact.
Let then $(f^k)_{k \geq 1} \subset N_{\tau}$ be such that $\norm{f^k}_H \leq 1$ for every $k \geq 1$.
Let $(z^{1,k}_n)_{n \geq 1} \subset \dom{A^*}$ be an associated sequence.
In particular, for every $k \geq 1$, there exists $n_k \geq 1$ such that, denoting by $w^{1,k}=z^{1,k}_{n_k}$, we have
$$\norm{S(\tau)^*w^{1,k}-f^k}_H \leq \frac{1}{k}, \quad \norm{B^*S(\cdot)^*w^{1,k}}_{L^2(0,\tau;U)} \leq \frac{1}{k}, \quad \forall k \geq 1.$$
Since $(f^k)_{k \geq 1}$ is bounded, so is $(S(\tau)^*w^{1,k})_{k \geq 1}$.
Using the same reasoning as in Step \ref{step 1}, we deduce from the estimates \eqref{obs ineq plus cpct 2 bis} and \eqref{obs ineq plus cpct 1 bis} that $(S(\tau)^*w^{1,k})_{k \geq 1}$ is a Cauchy sequence.
It follows that $(f^k)_{k \geq 1}$ is a Cauchy sequence as well, and thus converges.

\item
The next step is to establish that
\begin{equation}\label{NT is almost stable by A}
N_{\tau} \subset \dom{A^*}, \quad A^*(N_{\tau}) \subset N_{\tau-\epsilon}, \quad \forall \tau \in (T_0,T), \, \forall \epsilon \in (0,\tau-T_0).
\end{equation}
Let then $f \in N_{\tau}$.
By definition, there exists a sequence $(z^1_n)_{n \geq 1} \subset \dom{A^*}$ such that
\begin{gather}
S(\tau)^*z^1_n \xrightarrow[n \to +\infty]{} f \text{ in } H, \label{CV to f}
\\
B^*S(\cdot)^*z^1_n \xrightarrow[n \to +\infty]{} 0 \text{ in } L^2(0,\tau;U). \label{CV to zero}
\end{gather}
Using the estimate \eqref{obs ineq plus cpct 3 bis} with $t_1=T-\tau$ and $t_2=T-(\tau-\epsilon)$, we see that
$$(S(\tau-\epsilon)^*z^1_n)_{n \geq 1} \text{ is bounded in } H.$$
As before, it follows from the estimates \eqref{obs ineq plus cpct 2 bis} and \eqref{obs ineq plus cpct 1 bis} that there exists $g \in H$ such that
\begin{equation}\label{CV to g}
S(\tau-\epsilon)^*z^1_n \xrightarrow[n \to +\infty]{} g \text{ in } H.
\end{equation}
Noting \eqref{CV to f}, by uniqueness of the limit, we have
$$f=S(\epsilon)^*g.$$
Let us now prove that $g \in \dom{A^*}$.
By definition of the domain of the generator of a semigroup, we have to show that, for any sequence $t_n>0$ with $t_n \to 0$ as $n \to +\infty$, the sequence
$$u_n=\frac{S(t_n)^*g-g}{t_n}$$
converges in $H$ as $n \to +\infty$ and that its limit does not depend on the sequence $(t_n)_n$.
Let $n_0 \geq 1$ be large enough so that $t_n \leq \epsilon$ for every $n \geq n_0$.
From \eqref{CV to g} and \eqref{CV to zero} we easily see that
\begin{equation}\label{stab semi g}
S(t)^*g \in N_{\tau-\epsilon}, \quad \forall t \in [0,\epsilon].
\end{equation}
Thus,
$$u_n \in N_{\tau-\epsilon}, \quad \forall n \geq n_0.$$
Let now $\mu \in \rho(A^*) \neq \emptyset$ be fixed and let us introduce the following norm on $N_{\tau-\epsilon}$:
$$\norm{z}_{-1}=\norm{(\mu-A^*)^{-1}z}_H.$$
Since $(\mu-A^*)^{-1}g \in \dom{A^*}$, we have
\begin{equation}\label{un CV in diff norm}
(\mu-A^*)^{-1}u_n
=\frac{S(t_n)^*-\Id}{t_n}(\mu-A^*)^{-1}g \xrightarrow[n \to +\infty]{}
A^*(\mu-A^*)^{-1}g \quad \mbox{ in } H.
\end{equation}
Therefore, $(u_n)_{n \geq n_0}$ is a Cauchy sequence in $N_{\tau-\epsilon}$ for the norm $\norm{\cdot}_{-1}$.
Since $N_{\tau-\epsilon}$ is finite dimensional (reall \eqref{NT finite dim}), all the norms are equivalent on $N_{\tau-\epsilon}$.
Thus, $(u_n)_{n \geq n_0}$ is a Cauchy sequence for the usual norm $\norm{\cdot}_H$ as well and, as a result, converges for this norm.
It is clear from \eqref{un CV in diff norm} that its limit does not depend on the sequence $(t_n)_n$.
This shows that $g \in \dom{A^*}$ and thus $f=S(\epsilon)^*g \in \dom{A^*}$.
In addition, we have
$$A^*f=A^*S(\epsilon)^*g=\lim_{h \to 0^+} \frac{S(\epsilon)^*g-S(\epsilon-h)^*g}{h} \in N_{\tau-\epsilon} \quad (\text{by \eqref{stab semi g}}).$$
This shows that $A^*(N_\tau) \subset N_{\tau-\epsilon}$.

\item
Let us now prove that
$$N_{\tau} \subset \ker B^*, \quad \forall \tau \in (T_0,T).$$
Let $\epsilon \in (0,\tau-T_0)$ be arbitrary.
We use the same notations as in the previous step.
Since, by assumption, $H$ is an admissible subspace for $(A,B)$ (see Definition \ref{def admi}), the map $z^1 \in \dom{A^*} \mapsto B^*S(\epsilon-\cdot)^*z^1 \in L^2(0,\epsilon;U)$ can be extended to a bounded linear operator $\Psi \in \lin{H,L^2(0,\epsilon;U)}$.
From \eqref{CV to g} and continuity of $\Psi$, we have
$$\Psi S(\tau-\epsilon)^*z^1_n \xrightarrow[n \to +\infty]{} \Psi g, \quad \text{ in } L^2(0,\epsilon;U).$$

Since $z^1_n \in \dom{A^*}$, we have $(\Psi S(\tau-\epsilon)^*z^1_n)(t)=B^*S(\tau-t)z^1_n$ for $t \in (0,\epsilon)$.
From \eqref{CV to zero} and uniqueness of the limit, we deduce that
$$\Psi g=0.$$
Since $g \in \dom{A^*}$, we have $(\Psi g)(t)=B^*S(\epsilon-t)^*g$ and the map $t \in [0,\epsilon] \mapsto B^*S(\epsilon-t)^*g$ is continuous.
It follows that
$$B^*f=B^*S(\epsilon)^*g=(\Psi g)(0)=0.$$

\item
Next, we observe that there exist $\tau \in (T_0,T)$ and $\epsilon \in (0,\tau-T_0)$ such that
\begin{equation}\label{NT inv}
N_{\tau}=N_{\tau-\epsilon}.
\end{equation}
Indeed, from \eqref{NT finite dim} and \eqref{NT decreases}, the sequence of integers $(\dim N_{T-(T-T_0)/k})_{k \geq 2}$ is non-increasing and thus stationary: there exists $k_0 \geq 2$ such that
$$\dim N_{T-(T-T_0)/k}=\dim N_{T-(T-T_0)/k_0}, \quad \forall k \geq k_0.$$
Denoting by $\delta=(T-T_0)/k_0 \in (0,T-T_0)$ and using \eqref{NT decreases}, we then have
$$N_{\tau}=N_{T-\delta}, \quad \forall \tau \in [T-\delta,T).$$
The desired claim easily follows.

\item
Consequently, for $\tau \in (T_0,T)$ and $\epsilon \in (0,\tau-T_0)$ fixed such that \eqref{NT inv} holds, the restriction of $A^*$ to $N_{\tau}$ is a linear operator from the finite dimensional space $N_{\tau}$ into itself (recall \eqref{NT is almost stable by A}).
Besides, $N_{\tau} \neq \ens{0}$ since it contains $N_T$ by \eqref{NT decreases} and we have \eqref{NT not zero}.
Therefore, this restriction has at least one eigenvalue (recall that $H$ is a complex Hilbert space), i.e. there exist $\lambda \in \C$ and a nonzero $\phi \in N_{\tau}$ such that
$$A^* \phi=\lambda \phi.$$
Since $N_{\tau} \subset \ker B^*$, we also have $\phi \in \ker B^*$ and this is a contradiction with the Fattorini-Hautus test \eqref{FH test}.

\end{enumerate}

This concludes the proof of Theorem \ref{thm compactness}.

\begin{remark}
Let us stress that the end of our proof differs from the one in \cite[Section 2.2]{CN21}.
Indeed, in this reference, the conclusion of the proof relied on the fact that the semigroup is nilpotent, that is
$$\exists T>0, \quad S(T)^*z^1=0, \quad \forall z^1 \in H.$$
This readily implies that the operator $A^*$ has no eigenvalues and this is how the authors conclude that $N_T=\ens{0}$.
On the other hand, in our proof above, we only made use of the Fattorini-Hautus test \eqref{FH test} (which is trivially checked if the operator $A^*$ has no eigenvalues).
Besides, this is optimal, in the sense that this test is always a necessary condition for the system $(A,B)$ to be null controllable in some time.

Finally, let us add that for the example of the hyperbolic system $(\Lambda,-,Q,G)$ the corresponding adjoint semigroup is not always nilpotent.
Notably, the strictly lower triangular structure of $G_{--}$ was used at the end of \cite[Section 2.2]{CN21} to prove such a property.
\end{remark}

\bibliographystyle{amsalpha}
\bibliography{biblio}

\end{document}